\documentclass[11pt,letterpaper]{amsart}
\usepackage{graphicx}
\usepackage{amsmath}
\usepackage{mathtools}
\usepackage{amssymb}
\usepackage{amsfonts}
\usepackage{amsthm}
\usepackage{cancel}
\usepackage[all]{xy}
\usepackage[neveradjust]{paralist}
\usepackage{enumitem}
\usepackage{multicol}
\usepackage{mathrsfs}
\usepackage{mathdots}
\usepackage[pagebackref,colorlinks=true,linkcolor=blue,citecolor=blue]{hyperref}
\usepackage{tikz-cd}
\usepackage{tikz}
\usepackage{blkarray}
\usepackage{scalerel,stackengine}

\UseRawInputEncoding

\stackMath
\newcommand\reallywidehat[1]{%
\savestack{\tmpbox}{\stretchto{%
  \scaleto{%
    \scalerel*[\widthof{\ensuremath{#1}}]{\kern-.6pt\bigwedge\kern-.6pt}%
    {\rule[-\textheight/2]{1ex}{\textheight}}
  }{\textheight}%
}{0.5ex}}%
\stackon[1pt]{#1}{\tmpbox}%
}

\tikzset{
  symbol/.style={
    draw=none,
    every to/.append style={
      edge node={node [sloped, allow upside down, auto=false]{$#1$}}}
  }
} 

\usepackage{color}
\usepackage{blkarray}

\newcommand{\Z}{\mathbb{Z}}

\newcommand{\R}{\mathbb{R}}
\newcommand{\BC}{\mathbb{C}}

\newcommand{\A}{{\mathbb A}}

\newcommand{\inv}{^{-1}}

\newcommand{\SL}{\mathrm{SL}}
\newcommand{\GL}{\mathrm{GL}}
\newcommand{\SO}{\mathrm{SO}}
\newcommand{\Sp}{\mathrm{Sp}}
\newcommand{\St}{\mathrm{St}}

\newcommand{\Rep}{\underline{\mathrm{Rep}}}

\newcommand{\half}[1]{\frac{#1}{2}}

\newcommand{\comment}[1]{}
\newcommand{\EE}{\mathcal{E}}
\newcommand{\FF}{\mathcal{F}}

\newtheorem{thm}{Theorem}[section]
\newtheorem{cor}[thm]{Corollary}
\newtheorem{lemma}[thm]{Lemma}
\newtheorem{prop}[thm]{Proposition}
\newtheorem {conj}[thm]{Conjecture}

\newtheorem {ques/conj}[thm]{Question/Conjecture}

\newtheorem{defn}[thm]{Definition}
\newtheorem{remark}[thm]{Remark}

\newtheorem{algo}[thm]{Algorithm}

\DeclareMathOperator{\supp}{supp}

\numberwithin{equation}{section}

\let\oldbullet\bullet
\renewcommand{\bullet}{{\vcenter{\hbox{\tiny$\oldbullet$}}}}

\begin{document}
\renewcommand{\theequation}{\arabic{equation}}
\numberwithin{equation}{section}

\title[Enhanced Shahidi conjecture]
{On the enhanced Shahidi conjecture and global applications}

\author{Alexander Hazeltine
}
\address{Department of Mathematics\\
University of Michigan\\
Ann Arbor, MI 48109, USA}
\email{ahazelti@umich.edu}

\author{Baiying Liu}
\address{Department of Mathematics\\
Purdue University\\
West Lafayette, IN, 47907, USA}
\email{liu2053@purdue.edu}

\author{Chi-Heng Lo}
\address{Department of Mathematics\\
Purdue University\\
West Lafayette, IN, 47907, USA}
\email{lo93@purdue.edu}





\thanks{The research of the second-named author has been partially supported by the NSF Grants DMS-1702218, DMS-1848058}

\begin{abstract}
In this paper, applying the intersection theory of local Arthur packets, for symplectic and split odd special orthogonal groups $\mathrm{G}_n$, we give the first complete proof of the enhanced Shahidi conjecture on generic representations in local Arthur packets. We also classify unramified representations of Arthur type for $\mathrm{G}_n$, and show that they lie in exactly one local Arthur packet, which is anti-generic. Then, we discuss the global applications of these results.
\end{abstract}

\maketitle


\section{Introduction}

Let $k$ be a number field and $\mathbb{A}_k$ be the ring of adeles. 
Let $\mathrm{G}$ be a connected reductive group defined over $k$.
A main theme in the theory of automorphic forms is to study the discrete spectrum of $\mathrm{G}(\mathbb{A}_k)$. In \cite[\S 6, \S 8]{Art89}, Arthur conjectured that there are theories of local and global Arthur packets for $\mathrm{G}$, parameterized by local and global Arthur parameters, such that the discrete spectrum of $\mathrm{G}(\mathbb{A}_k)$ is a union of global Arthur packets.

There are many conjectures related to the local Arthur packets. In this paper, 
applying the results in \cite{HLL22} on the intersection theory of local Arthur packets, 
we focus on the following problems for symplectic groups and quasi-split odd special orthogonal groups.
\begin{enumerate}
    \item The enhanced Shahidi conjecture on generic representations in local Arthur packets for quasi-split connected reductive groups and its global applications. 
    \item Classification of unramified representations of Arthur type and its global applications.  
\end{enumerate}

\subsection{Global and local Arthur packets}
To proceed with the discussion, let us recall the notation for global and local Arthur packets. Let $L_k$ be the hypothetical Langlands group. The set of global Arthur parameters $\Psi(\mathrm{G}(\A_k))$ consists of admissible homomorphisms
\begin{equation*}
    \psi: L_k \times \mathrm{SL}_2^A(\mathbb{C}) \rightarrow {}^L\mathrm{G}(k),
\end{equation*}
where ${}^L\mathrm{G}(k)=\widehat{\mathrm{G}}(\mathbb{C}) \rtimes W_k$ is the global Langlands $L$-group of $\mathrm{G}$, $\widehat{\mathrm{G}}(\mathbb{C})$ is the Langlands dual group of $\mathrm{G}$ and $W_k$ is the Weil group of $k$.  
Denote the conjectural global Arthur packet corresponding to a global Arthur parameter $\psi$ by $\Pi_{\psi}$. For classical groups, the precise description of $\Psi(\mathrm{G}(\A_k))$ without referring to the hypothetical Langlands group is given in \S \ref{sec global Arthur packets}.

For each local place $v$ of $k$, the set of local Arthur parameters $\Psi^+(\mathrm{G}(k_v))$ obtained from restricting global Arthur parameters to 
$$L_{k_v} = \begin{cases}
   W_{k_v}, & v \text{ Archimedean,}\\
W_{k_v} \times \mathrm{SL}_2^D(\mathbb{C}), & v \text{ non-Archimedean,}
\end{cases}$$
where $W_{k_v}$ is the Weil group of $k_v$. Thus, $\Psi^+(\mathrm{G}(k_v))$ consists of admissible homomorphisms
\begin{equation*}
    \psi_v: L_{k_v} \times \mathrm{SL}_2^A(\mathbb{C}) \rightarrow {}^L\mathrm{G}(k_v),
\end{equation*}
where ${}^L\mathrm{G}(k_v)=\widehat{\mathrm{G}}(\mathbb{C}) \rtimes W_{k_v}$ is the local Langlands $L$-group of $\mathrm{G}$. If we write 
\begin{align}\label{eq psi_v intro}
    \psi_v |_{W_{k_v}} = \bigoplus_i \phi_i|\cdot|^{x_i},
\end{align}
where
$\phi_i$ is an irreducible representation of $W_{k_v}$ with bounded image consisting of semi-simple elements and $x_i \in \R$, then $|x_i|<\half{1}$. Let $\Psi(\mathrm{G}(k_v))$ be the subset of local Arthur parameters with $x_i=0$ for all $i$ and let $\Pi_{\psi_v}$ denote the conjectural local Arthur packet corresponding to  $\psi_v$.

We need more notations for local Arthur parameters. A local Arthur parameter $\psi_v$ is called {\it generic} if $\psi_v|_{\mathrm{SL}_2^A(\mathbb{C})}$ is trivial; and is called {\it tempered} if additionally $x_i=0$ for all $i$ in \eqref{eq psi_v intro}. For each local Arthur parameter $\psi_v$, Arthur associated a local $L$-parameter $\phi_{\psi_v}$ as follows
\begin{equation}\label{apequ1}
\phi_{\psi_v}(w, x) = \psi_v\left(w, x, \begin{pmatrix}
        |w|^{\frac{1}{2}} & 0 \\
        0 & |w|^{-\frac{1}{2}}\\
\end{pmatrix}\right).
\end{equation}
An irreducible admissible representation of $\mathrm{G}(k_v)$ is of {\it Arthur type} if it lies in $\Pi_{\psi_v}$ for some $\psi_v \in \Psi^+(\mathrm{G}(k_v))$.

Arthur's conjectures state that the discrete spectrum of automorphic representations of $\mathrm{G}(\A_k)$ is partitioned into global Arthur packets. Given any global Arthur parameter $\psi$, we have its localization $\psi_v \in \Psi^+(\mathrm{G}(k_v))$ via restriction to $L_{k_v}$ for any place $v$. Then, for any $\pi \in \Pi_{\psi}$, with $\pi=\otimes_v \pi_v$, we have $\pi_v \in \Pi_{\psi_v}$. In \cite{Art13}, Arthur proved these global conjectures and the existence of local Arthur packets for symplectic and quasi-split special orthogonal groups.

\subsection{The enhanced Shahidi conjecture}\label{sec: enhanced Shahidi conjecture intro}

Let $\mathrm{G}$ be a quasi-split connected reductive group defined over a non-Archimedean local field $k_v$ and let $G=\mathrm{G}(k_v)$. 
The well-known Shahidi conjecture states that tempered local $L$-packets of $G$ have generic members (\cite[Conjecture 9.4]{Sha90}). Recently, Shahidi made an enhanced conjecture as follows.

\begin{conj}[{\cite[Conjecture 1.5]{LS22}}, Enhanced Shahidi Conjecture]\label{Enhanced Shahidi conjecture intro}
For any quasi-split connected 
 reductive group $G$,
assume that there is a theory of local Arthur packets for $G$ as conjectured in \cite[\S 6]{Art89}. 
\begin{enumerate}
    \item For any local Arthur parameter $\psi \in \Psi(G)$, the local Arthur packet $\Pi_{\psi}$ is tempered if and only if it has a generic member.
    \item For any local Arthur parameter $\psi \in \Psi^+(G)$, the local Arthur packet $\Pi_{\psi}$ is generic if and only if it has a generic member.
\end{enumerate}
\end{conj}

In \cite{LS22}, the second-named author and Shahidi proved Conjecture \ref{Enhanced Shahidi conjecture intro} for symplectic and quasi-split special orthogonal groups, applying the matching method of endoscopic liftings, assuming properties of the wavefront sets of certain bitorsor representations. 

Let $\mathrm{G}_n$ be the symplectic or the split odd special orthogonal group of rank $n$ and let $G_n=\mathrm{G}_n(k_v)$. 
As an application of the authors' results on the intersection theory of local Arthur packets in \cite{HLL22} and the classification of generic representations as in \cite{Mui98, JS04, Liu11}, we prove Conjecture \ref{Enhanced Shahidi conjecture intro} for $G_n$, without any assumptions.

\begin{thm}[Theorems \ref{proof of enh sha conj} and \ref{thm non-unitary Enhanced Shahidi}]
\label{proof of enh sha conj intro}
 Conjecture \ref{Enhanced Shahidi conjecture intro} is valid for $G_n$. 
\end{thm}

We remark that the proof of Theorem \ref{proof of enh sha conj intro} is extracted from the first arXiv version of \cite{HLL22} and is the {\it first complete} proof of the enhanced Shahidi conjecture. The method used here is different from that in \cite{LS22}.
Later, in \cite{HLLZ22}, jointly with Zhang, we gave another proof for $G_n$, applying the closure ordering relations among the local $L$-parameters of representations in local Arthur packets, which is geometric in nature.

Since generic local Arthur packets are disjoint, Theorem \ref{proof of enh sha conj intro} implies that a (local) generic representation of Arthur type lies in exactly one local Arthur packet, whose parameter is generic. As an application of this observation, we obtain the following result on local components of automorphic representations in global Arthur packets with a generic local component at some finite place. 
This result has its own interest and is expected by experts. 

\begin{prop}[Proposition \ref{thm partial Clozel 4}]\label{thm partial Clozel 4 intro}
    Suppose that $\pi=\otimes\pi_v$ lies in a global Arthur packet $\Pi_\psi$ of $\mathrm{G}_n(\mathbb{A}_k)$ and that there exists a finite place $v_0$ such that $\pi_{v_0}$ is generic. Then the following holds:
    \begin{enumerate}
        \item The global Arthur parameter is of the form $\psi=\boxplus_{i=1}^m \mu_i\otimes S_1,$ where $\mu_i$'s are irreducible cuspidal representations of general linear groups.
        \item For any place $v,$  $\psi_v$ is generic. In particular, $\pi_{v}$ is generic for almost all finite places $v$.
        \item Furthermore, for almost all places $v$, the  local Arthur parameter $\psi_v$ has trivial restrictions to both the Deligne-$\SL_2(\mathbb{C})$ and Arthur-$\SL_2(\mathbb{C}).$
    \end{enumerate}
\end{prop}

We conjecture that Part (2) of the above proposition holds for general connected reductive groups.

\begin{conj}
\label{conj generic almost all intro}
    Let $\pi$ be an automorphic representation in the discrete spectrum of a connected reductive group $\mathrm{G}(\mathbb{A}_k)$. Suppose there exists a finite place $v_0$ of $k$ such that $\pi_{v_0}$ is generic, then $\pi_v$ is generic for almost all places.
\end{conj}

In the proof of Proposition \ref{thm partial Clozel 4 intro}, we do not assume the Ramanujan conjecture of $\GL_n$. On the other hand, if we do assume the Ramanujan conjecture of $\GL_n$, then we obtain a stronger conclusion. 

\begin{prop}[Proposition \ref{thm stronger condition}]\label{thm stronger condition intro}
    Assume the Ramanujan conjecture for $\GL_n$. Suppose that $\pi=\otimes\pi_v$ lies in a global Arthur packet $\Pi_\psi$ of $\mathrm{G}_{n}(\A_k)$ and that there exists a finite place $v_0$ such that $\pi_{v_0}$ is generic. Then $\pi_v$ is tempered for all places $v$. 
\end{prop}

This result is towards the generalized Ramanujan problem, as posted by Sarnak in \cite{Sar05}, to characterize the failure of temperedness of local components of general automorphic representations.
 Note that, applying Langlands functoriality, it is known that the Ramanujan conjecture for $\GL_n$ implies the generalized Ramanujan conjecture for quasi-split classical groups $\mathrm{G}$, that is, if $\pi=\otimes_v \pi_v$ is a globally generic cuspidal automorphic representation of $\mathrm{G}(\A_k)$, then $\pi_v$ are all tempered (see \cite[(6.6)]{Sha11}). Proposition \ref{thm stronger condition intro} improves the results in \cite[Theorem 6.1 and Conjecture 6.5]{Sha11}, where locally generic (at every local place) condition is assumed. 

We remark that the following conjecture of Clozel is a direct corollary of Proposition \ref{thm stronger condition intro} under the assumption of the Ramanujan conjecture for $\GL_n$.

\begin{conj}[{\cite[Conjecture 5]{Clo07}}]
\label{conj Clozel 5 intro}
     Let $\pi$ be an automorphic representation in the discrete spectrum of a reductive group $\mathrm{G}(\A_k)$. Suppose there exists a finite place $v_0$ of $k$ such that $\pi_{v_0}$ is a Steinberg representation, then $\pi_v$ is tempered for all places $v$.
\end{conj}

\subsection{Unramified representations of Arthur type and global applications}
Considering the importance of unramified representations in the theory of automorphic forms and automorphic representations, it is desirable to study more closely the unramified representations of Arthur type. 
The following dual version of the enhanced Shahidi conjecture is expected by experts. For quasi-split classical groups, this is a consequence of
\cite[Proposition 6.4]{Moe09b}. 

\begin{conj}\label{conj unram intro}
Let $\mathrm{G}$ be a connected reductive group defined over a non-Archimedean local field $k_v$ and let $G=\mathrm{G}(k_v)$. Assume that there is a theory of local Arthur packets for $G$ as conjectured in \cite[\S 6]{Art89}. Then the following holds.
\begin{enumerate}
    \item [(i)] Any unramified representation of $G$ of Arthur type lies in exactly one local Arthur packet. Moreover, it lies in the $L$-packet associated to an anti-generic local Arthur packet.
    \item [(ii)]$($\cite[Conjecture 2A]{Clo07}$)$  Let $\psi \in \Psi^+(G)$. If $\phi_\psi$ is unramified, then the local Arthur packet $\Pi_\psi$ contains a unique unramified representation. More specifically, the unramified representation is the one associated to $\phi_\psi$ via the Satake isomorphism $($\cite{Sat63}$)$.
\end{enumerate}

\end{conj}

In \S \ref{sec unram classification}, we give a classification of unramified representations of $G_n$ of Arthur type in terms of their $L$-data (see Theorem \ref{prop unram L-data}), making use of Algorithm \ref{alg Arthur type}. Then, we give a new proof of Conjecture \ref{conj unram intro} for $G_n$ (see Proposition \ref{cor unram dual} and Corollary \ref{cor unram at most one}). We remark that Conjecture \ref{conj unram intro} directly implies the following expectation of Arthur and Clozel (see \cite[Conjecture 6.1]{Sha11}).

\begin{cor}\label{thm unram intro 3}
Let $\pi=\otimes_v \pi_v$ be an automorphic representation of $\mathrm{G}_n(\mathbb{A}_k)$ in the discrete spectrum with global Arthur parameter $\psi \in \Psi(\mathrm{G}_n(k))$. 
Then, for almost all finite places $v$, we have $\pi_v\in\Pi_{\phi_{\psi_v}}.$ 
\end{cor}

Clozel also made the following deep conjectures on local components of automorphic representations in the discrete spectrum of connected reductive groups $\mathrm{G}$, towards 
 the generalized Ramanujan problem as posted by Sarnak.
 
\begin{conj}[{\cite[Conjectures 2, 4]{Clo07}}]\label{conj Clozel 2 intro}
Let $\mathrm{G}$ be a connected reductive group defined over a number field $k$. 
Let $\pi=\otimes_v \pi_v$ be an automorphic representation of $\mathrm{G}(\mathbb{A}_k)$ in the discrete spectrum.
\begin{enumerate}
    \item For any finite place $v_0$ such that $\pi_{v}$ is unramified, the Satake parameter of $\pi_v$ is of the form $\phi_{\psi_v}(\mathrm{Frob}_{v}) $ for some $\psi_v \in \Psi(\mathrm{G}(k_v))$. Moreover, write the multi-set of the absolute value of eigenvalues of $\phi_{\psi_v}(\mathrm{Frob}_v)$ as $\{q_v^{w_{v,1}},\ldots, q_v^{w_{v, N}}\}$, where $q_v$ is the cardinality of the residue field of $k_{v}$. Then the multi-set $\{w_{v,1},\ldots, w_{v, N}\}$ is independent of the unramified place $v$.
    \item If there exists a finite place ${v_0}$ such that $\pi_{v_0}$ is unramified and tempered, then every component of $\pi$ is tempered.
\end{enumerate}
\end{conj}

We remark that a key point of Part (1) is that in the statement, $\psi_v \in \Psi(\mathrm{G}(k_v))$, not $\Psi^+(\mathrm{G}(k_v))$. 
As applications of the results in \S \ref{sec unram classification}, we show that the Ramanujan conjecture of $\GL_n$ implies Conjecture \ref{conj Clozel 2 intro} for the groups $\mathrm{G}_n$ (see Proposition \ref{thm conj Clozel 2}). This illustrates the depth and difficulties of these conjectures. Clozel originally made these conjectures for $L^2$-automorphic representations, here for simplicity, we only consider the case of discrete spectrum. Note that by Langlands' theory of Eisenstein series, the corresponding conjectures for $L^2$-automorphic representations reduce to those for the discrete spectrum.

Following is the structure of this paper. In \S \ref{sec nota and prel}, we recall necessary notation and preliminaries. In \S \ref{sec atob refor}, we recall Atobe's reformulation on M{\oe}glin's construction of local Arthur packets and certain results from the intersection theory of local Arthur packets. 
In \S \ref{section enhanced Shahidi conjecture}, we study the existence of generic representations in local Arthur packets and prove Conjecture \ref{Enhanced Shahidi conjecture intro} for $\mathrm{G}_n$. 
In \S \ref{sec unram classification}, we classify unramified representations of $\mathrm{G}_n$ of Arthur type in terms of their $L$-data, and we verify Conjecture \ref{conj unram intro} for $\mathrm{G}_n$. In \S \ref{sec global}, we discuss the global implications of the results in \S \ref{section enhanced Shahidi conjecture} and \S \ref{sec unram classification}.

\subsection*{Acknowledgements} 
The authors would like to thank Dihua Jiang and Freydoon Shahidi for their interests, constant support, helpful discussions, and for bringing our attention on the topics of unramified representations in \S \ref{sec unram classification} and \S \ref{sec global}. 
The authors also would like to thank Wee Teck Gan, Tasho Kaletha, Anantharam Raghuram, David Renard, Bin Xu, and Lei Zhang for helpful comments and suggestions.

\section{Notation and preliminaries}\label{sec nota and prel}

Let $F$ be a local non-Archimedean field of characteristic $0$ with normalized absolute value given by $|\cdot|.$ We also regard $|\cdot|$ as a character of $\GL_n(F)$ via composition with the determinant. Set $G_n$ to be the split group $\SO_{2n+1}(F)$ or $\Sp_{2n}(F).$ We write $\Pi(G)$ for the set of equivalence classes of irreducible smooth representations of a group $G.$ We assume that every representation is smooth. 

Suppose $\Pi_1, \Pi_2$ are representations of finite length. We let $[\Pi_1]$ denote the image of $\Pi_1$ in the Grothendieck group. We write $\Pi_1 \geq \Pi_2$ if $[\Pi_1]-[\Pi_2]$ is a non-negative linear combination of irreducible representations.

For a multi-set $X$ and $a\in X$, we let $m_X(a)$ denote the multiplicity of $a$ in $X.$ Let $Y$ be another multi-set. We define the following multi-sets $Z$ by specifying the multiplicity $m_Z(a)$ for each $a$ such that $m_X(a)+m_{Y}(a)>0$.
\begin{enumerate}
    \item [$\oldbullet$] The sum of multi-sets $X+Y$: $m_{X+Y}(a):=m_X(a)+m_Y(a) $.
    \item [$\oldbullet$] The union of multi-sets $X\cup Y$: $m_{X\cup Y}(a):=\max(m_X(a),m_Y(a)).$
    \item [$\oldbullet$] The difference of multi-sets $X\setminus Y$: $m_{X\setminus Y}(a):=\max(m_X(a)-m_Y(a),0).$
    \item [$\oldbullet$] The intersection of multi-sets $X \cap Y$: $m_{X\cap Y}(a):=\min(m_X(a),m_Y(a)).$
    \item [$\oldbullet$] The symmetric difference of multi-sets $X\Delta Y:= (X\cup Y)\setminus(X\cap Y).$
\end{enumerate}

\subsection{Langlands classification}
In this subsection, we recall the Langlands classification for the groups $\GL_n(F)$ and $G_n$ (see \cite{Kon03} for a general setting).

First, we consider $\GL_n(F)$. Fix a Borel subgroup of $\GL_n(F)$. Let $P$ be a standard parabolic subgroup of $\GL_n(F)$ with Levi subgroup $M\cong \GL_{n_1}(F)\times\cdots\times\GL_{n_r}(F).$ Let $\tau_i\in \Pi(\GL_{n_i}(F))$ for $i=1,2,\dots,r.$ We set
$$
\tau_1\times\cdots\times\tau_r := \mathrm{Ind}_{P}^{\GL_n(F)}(\tau_1\otimes\cdots\otimes\tau_r)
$$
to be the normalized parabolic induction. A segment $[x,y]_\rho$ is a set of supercuspidal representations of the form
$$
[x,y]_\rho=\{\rho|\cdot|^x, \rho|\cdot|^{x-1},\dots,\rho|\cdot|^y\},
$$
where $\rho$ is an irreducible supercuspidal representation of $\GL_n(F)$ and $x,y\in\mathbb{R}$ such that $x-y$ is a non-negative integer. Let $\Delta_{\rho}[x,y]$ be the Steinberg representation attached to the segment $[x,y]_\rho$, which is the unique irreducible subrepresentation of $\rho|\cdot|^x\times\cdots\times\rho|\cdot|^y.$ It is an essentially discrete series representation of $\GL_{n(x-y+1)}(F).$ We also let $Z_\rho[y,x]$ denote the unique irreducible quotient of $\rho|\cdot|^x\times\cdots\times\rho|\cdot|^y.$ If $y=x+1,$ we set $\Delta_\rho[x,x+1]=Z_\rho[x+1,x]$ to be the trivial representation of $\GL_0(F).$

The Langlands classification for $\GL_n(F)$ states that any irreducible representation $\tau$ of $\GL_n(F)$ can be realized as a unique irreducible subrepresentation of a parabolic induction of the form
\[\Delta_{\rho_1}[x_1,y_1]\times\cdots\times\Delta_{\rho_r}[x_r,y_r],\]
where $\rho_i$ is an irreducible unitary supercuspidal representation of $\GL_{n_i}(F),$ $[x_i,y_i]_{\rho_i}$ is a segment, and $x_1+y_1\leq\cdots\leq x_r+y_r.$ In this setting, we write
$$
\tau=L(\Delta_{\rho_1}[x_1,y_1],\dots,\Delta_{\rho_r}[x_r,y_r]).
$$

Let $(x_{i,j})_{1\leq i\leq s, 1\leq j \leq t}$ be real numbers such that $x_{i,j}=x_{1,1}-i+j.$ Define a \emph{(shifted) Speh representation} to be the irreducible representation given by
$$
\begin{pmatrix}
x_{1,1} & \cdots & x_{1,t} \\
\vdots & \ddots & \vdots \\
x_{s,1} & \cdots & x_{s,t}
\end{pmatrix}_{\rho}:=L(\Delta_{\rho}[x_{1,1},x_{s,1}],\dots,\Delta_{\rho}[x_{1,t}, x_{s,t}]).
$$

Fix a Borel subgroup of $G_n$ and let $P$ be a standard parabolic subgroup of $G_n$ with Levi subgroup $M\cong\GL_{n_1}(F)\times\cdots\times\GL_{n_r}(F)\times G_{m}.$ Let $\tau_i$ be a representation of $\GL_{n_i}(F)$ for $i=1,2,\dots,r$ and $\sigma$ be a representation of $G_{m}$. We set
$$
\tau_1\times\cdots\times\tau_r\rtimes\sigma := \mathrm{Ind}_{P}^{G_n}(\tau_1\otimes\cdots\otimes\tau_r\otimes\sigma)
$$
to be the normalized parabolic induction. 

The Langlands classification for $G_n$ states that every irreducible representation $\pi$ of $G_n$ is a unique irreducible subrepresentation of 
\[\Delta_{\rho_1}[x_1,y_1]\times\cdots\times\Delta_{\rho_r}[x_r,y_r]\rtimes\pi_{temp},\]
where $\rho_i$ is an irreducible unitary supercuspidal representation of $\GL_{n_i}(F),$ $x_1+y_1\leq\cdots\leq x_r+y_r<0,$ and $\pi_{temp}$ is an irreducible tempered representation of $G_{m}.$ In this case, we write
$$
\pi=L(\Delta_{\rho_1}[x_1,y_1],\dots,\Delta_{\rho_r}[x_r,y_r];\pi_{temp}),
$$
and call $(\Delta_{\rho_1}[x_1,y_1],\dots,\Delta_{\rho_r}[x_r,y_r];\pi_{temp})$ the Langlands data, or $L$-data, of $\pi.$ In \S \ref{sec tempered parametrization}, We give more detailed parametrization of the tempered representation $\pi_{temp}$ using Arthur's theory. 

\subsection{Derivatives and socles}
\label{sec: der and socle}

Let $\pi$ be a smooth representation of $G_n$ of finite length. We let $Jac_{P}(\pi)$ be the Jacquet module of $\pi$ with respect to a parabolic subgroup $P$ of $G_n.$  Note that the semisimplification of $Jac_{P}(\pi)$ is given by $[Jac_{P}(\pi)].$

\begin{defn}
Let $P_d$ be a standard parabolic subgroup of $G_n$ with Levi subgroup isomorphic to $\GL_{d}(F)\times G_{n-d},$ $x\in\mathbb{R},$ and $\rho$ be an irreducible unitary self-dual supercuspidal representation of $\GL_d(F).$ Define the $\rho|\cdot|^x$-derivative of $\pi$, denoted $D_{\rho|\cdot|^x}(\pi),$ to be a semisimple representation satisfying
$$
[Jac_{P_d}(\pi)]=\rho|\cdot|^x\otimes D_{\rho|\cdot|^x}(\pi) + \sum_i \tau_i\otimes\pi_i,
$$
where the sum is taken over all irreducible representations $\tau_i$ of $\GL_d(F)$ such that $\tau_i\not\cong\rho|\cdot|^x$ and  $\pi_i$ are some representations of $G_{n-d}$.
\end{defn}

Set $D_{\rho|\cdot|^{x}}^{(0)}(\pi)=\pi$ and for any positive integer $k$, define $D_{\rho|\cdot|^{x}}^{(k)}(\pi)$ recursively by
$$
D_{\rho|\cdot|^{x}}^{(k)}(\pi)=\frac{1}{k}D_{\rho|\cdot|^{x}}\circ D_{\rho|\cdot|^{x}}^{(k-1)}(\pi).
$$
If $D_{\rho|\cdot|^{x}}^{(k)}(\pi)\neq 0$, but $D_{\rho|\cdot|^{x}}^{(k+1)}(\pi)=0,$ then $D_{\rho|\cdot|^{x}}^{(k)}(\pi)$ is called the \emph{highest $\rho|\cdot|^{x}$-derivative} of $\pi$. If $D_{\rho|\cdot|^{x}}(\pi)=0,$ then $\pi$ is called \emph{$\rho|\cdot|^{x}$-reduced}.

We also need the notion of derivatives for $\GL_n(F).$ However, in this situation, we must distinguish between left and right derivatives. We follow \cite[\S5]{Xu17a}.

\begin{defn}
Let $P_d$ $($resp. $Q_d$$)$ be a standard parabolic subgroup of $\GL_n(F)$ with Levi subgroup isomorphic to $\GL_{d}(F)\times \GL_{n-d}(F)$ $($resp. $\GL_{n-d}(F)\times \GL_{d}(F)$$)$, $x\in\mathbb{R},$ $\sigma$ be a smooth representation of $\GL_n(F)$, and $\rho$ be an irreducible unitary self-dual supercuspidal representation of $\GL_d(F).$ Define the left $($resp. right$)$ $\rho|\cdot|^x$-derivative of $\sigma$, denoted $D_{\rho|\cdot|^x}(\sigma)$ $($resp. $D_{\rho|\cdot|^x}^{op}(\sigma)$$)$, to be a semisimple representation satisfying
$$
[Jac_{P_d}(\sigma)]=\rho|\cdot|^x\otimes D_{\rho|\cdot|^x}(\sigma) + \sum_i \tau_i\otimes\sigma_i,
$$

$$
\left( \text{resp.  } [Jac_{Q_d}(\sigma)]= D_{\rho|\cdot|^x}^{op}(\sigma)\otimes\rho|\cdot|^x + \sum_i \sigma_i\otimes\tau_i, \right),
$$
where the sum is taken over all irreducible representations $\tau_i$ of $\GL_d(F)$ such that $\tau_i\not\cong\rho|\cdot|^x$ and $\sigma_i$ are some representations of $\GL_{n-d}(F).$
\end{defn}

Note that the right derivative defined in \cite[\S5]{Xu17a} uses the contragredient; however, for our purposes, we are only concerned with derivatives for self-dual $\rho.$ These derivatives satisfy the following Leibniz rules.

\begin{lemma}[{\cite[\S5]{Xu17a}}] \label{lem Leibniz rule}
Let $\rho$ be an irreducible unitary self-dual supercuspidal representation of $\GL_d(F)$ and $x\in\mathbb{R}.$
\begin{enumerate}
    \item [1.] For  $\sigma,\sigma_1,\sigma_2 \in \Pi(\GL_n(F)), \ \tau \in \Pi(G_m)$, we have
    \[ D_{\rho|\cdot|^x}( \sigma \rtimes \tau)= D_{\rho|\cdot|^x}(\sigma) \times \tau + D_{\rho|\cdot|^{-x}}^{op}(\sigma) \times \tau + \sigma \rtimes D_{\rho|\cdot|^x}(\tau),\]
    \begin{align*}
    D_{\rho|\cdot|^x}(\sigma_1 \times \sigma_2)&= D_{\rho|\cdot|^x}(\sigma_1) \times \sigma_2+ \sigma_1 \times D_{\rho|\cdot|^x}(\sigma_2),\\
    D_{\rho|\cdot|^{x}}^{op}(\sigma_1 \times \sigma_2)&= D_{\rho|\cdot|^{x}}^{op}(\sigma_1) \times \sigma_2+ \sigma_1 \times D_{\rho|\cdot|^{x}}^{op}(\sigma_2).
\end{align*} 
\item [2.] For $a\geq b$,
\begin{align*}
      D_{\rho|\cdot|^x}(\Delta_{\rho}[a,b])= \begin{cases} 0 & \text{ if } x \neq a,\\
\Delta_{\rho}[a-1,b] & \text{ if }x=a.\end{cases} \ \ \  &
      D_{\rho|\cdot|^x}^{op}(\Delta_{\rho}[a,b])= \begin{cases} 0 & \text{ if } x \neq b,\\
\Delta_{\rho}[a,b+1] & \text{ if }x=b.\end{cases}\\
D_{\rho|\cdot|^x}(Z_{\rho}[b,a])= \begin{cases} 0 & \text{ if } x \neq b,\\
Z_{\rho}[b+1,a] & \text{ if }x=b.\end{cases} \ \ \  &
      D_{\rho|\cdot|^x}^{op}(Z_{\rho}[b,a])= \begin{cases} 0 & \text{ if } x \neq a,\\
Z_{\rho}[b,a-1] & \text{ if }x=a.\end{cases} 
\end{align*}
\item [3.] $D_{\rho|\cdot|^{x}}$ commutes with $D_{\rho|\cdot|^{y}}$ if $|x-y|>1$. 
\end{enumerate}
\end{lemma}
For a multi-set of real numbers $\{x_1,\dots,x_r\}$, we denote the composition of derivatives by
\[ D_{\rho|\cdot|^{x_1,\dots ,x_r}}(\pi):=D_{\rho|\cdot|^{x_r}} \circ \cdots \circ D_{\rho|\cdot|^{x_1}}(\pi).\]
For example, if $\{x,\dots,x\}$ contains $k$ copies of $x$, then
\[ D_{\rho|\cdot|^{x,\dots,x} }(\pi)= (k!)\cdot D_{\rho|\cdot|^{x}}^{(k)}(\pi). \]

Let $\pi$ be a representation of finite length. Define the \emph{socle} of $\pi$, denoted by $soc(\pi)$, to be the maximal semisimple subrepresentation of $\pi.$

\begin{defn}
Let $\pi$ be a representation of finite length, $x\in\mathbb{R},$ and $\rho$ be an irreducible unitary self-dual supercuspidal representation of $\GL_d(F).$ Define
$$
S_{\rho|\cdot|^{x}}^{(r)}(\pi):=soc((\rho|\cdot|^{x})^r\rtimes\pi).
$$
\end{defn}
For a multi-set of real number $\{x_1,\dots,x_r\}$, we denote the composition of socles by
\[ S_{\rho|\cdot|^{x_1,\dots ,x_r}}(\pi):=S_{\rho|\cdot|^{x_r}} \circ \cdots \circ S_{\rho|\cdot|^{x_1}}(\pi).\]
\begin{thm}[{\cite[Lemma 3.1.3]{Jan14}, \cite[Propositions 3.3, 6.1, Theorem 7.1]{AM20}}]\label{thm derivative-socle}
Let $\rho$ be an irreducible unitary self-dual supercuspidal representation of $\GL_d(F),$ $\pi\in \Pi(G_n),$ and $x\in\mathbb{R}\setminus\{0\}.$ For any non-negative integers $k$ and $r$, we have the following. 
\begin{enumerate}
    \item The highest $\rho|\cdot|^{x}$-derivative of $\pi$, say $D_{\rho|\cdot|^{x}}^{(k)}(\pi),$ is irreducible.
    \item $S_{\rho|\cdot|^{x}}^{(r)}(\pi)$ is irreducible for any $r\geq 0.$
    \item We have
    $$
    S_{\rho|\cdot|^{x}}^{(k)}(D_{\rho|\cdot|^{x}}^{(k)}(\pi))=\pi,
    $$
    and, 
    $$
    D_{\rho|\cdot|^{x}}^{(k+r)}(S_{\rho|\cdot|^{x}}^{(r)}(\pi))=D_{\rho|\cdot|^{x}}^{(k)}(\pi).
    $$
    \item The $L$-data of $D_{\rho|\cdot|^{x}}^{(k)}(\pi)$ and $S_{\rho|\cdot|^{x}}^{(k)}(\pi)$ can be explicitly described in terms of those of $\pi.$
\end{enumerate}
\end{thm}

When $x=0,$ computing the $\rho$-derivative explicitly is generally difficult. As a remedy, Atobe and M{\'i}nguez considered the $\Delta_\rho[0,-1]$-derivative and $Z_\rho[0,1]$-derivative, denoted by $D_{\Delta_\rho[0,-1]}^{(k)}(\pi)$ and $D_{Z_\rho[0,1]}^{(k)}(\pi)$, respectively. These are semisimple representations of $G_{n-2dk}$  defined by
$$
[Jac_{P_{2dk}}(\pi)]=\Delta_\rho[0,-1]^k\otimes D_{\Delta_\rho[0,-1]}^{(k)}(\pi)+Z_\rho[0,1]^k\otimes D_{Z_\rho[0,1]}^{(k)}(\pi) + \sum_i \tau_i\otimes\pi_i,
$$
where the sum is taken over all irreducible representations $\tau_i$ of $\GL_{2dk}(F)$ such that $\tau_i$ is neither isomorphic to $\Delta_\rho[0,-1]^k$ nor $Z_\rho[0,1]^k.$ Similarly, consider
$$
S_{\Delta_\rho[0,-1]}^{(r)}(\pi):=soc(\Delta_\rho[0,-1]^r\rtimes\pi), \, \, S_{Z_\rho[0,1]}^{(r)}(\pi):=soc(Z_\rho[0,1]^r\rtimes\pi).
$$
These derivatives and socles satisfy analogous results as in Theorem \ref{thm derivative-socle}.

\begin{thm}[{\cite[Proposition 3.7]{AM20}}]\label{thm self-dual derivative}
Let $\rho$ be an irreducible unitary self-dual supercuspidal representation of $\GL_d(F)$ and $\pi\in \Pi(G_n).$ Assume that $\pi$ is $\rho|\cdot|\inv$-reduced $($respectively $\rho|\cdot|$-reduced$)$. Then the results of Theorem \ref{thm derivative-socle}(1), (2), and (3) hold with $\rho|\cdot|^x$ replaced by $\Delta_\rho[0,-1]$ $($respectively $Z_\rho[0,1])$.
\end{thm}

\subsection{Local Arthur packets}
Recall that a local Arthur parameter
$$\psi: W_F \times \SL^D_2(\mathbb{C}) \times \SL^A_2(\mathbb{C}) \rightarrow \widehat{G}_n(\BC)$$
is a direct sum of irreducible representations

\begin{equation}\label{eq decomp psi +}
  \psi = \bigoplus_{i=1}^r \phi_i|\cdot|^{x_i} \otimes S_{a_i} \otimes S_{b_i},  
\end{equation}
satisfying the following conditions:
\begin{enumerate}
    \item [(1)]$\phi_i$ is an irreducible representation of $W_F$ with bounded image consisting of semi-simple elements;
    \item [(2)] $x_i \in \R$ and $|x_i|<\half{1}$;
    \item [(3)]the restrictions of $\psi$ to the two copies of $\SL_2(\mathbb{C})$ are analytic, $S_k$ is the $k$-dimensional irreducible representation of $\SL_2(\mathbb{C})$, and, 
    $$\sum_{i=1}^r \dim(\phi_i)a_ib_i = N:= 
\begin{cases}
2n+1 & \text{ when } G_n=\Sp_{2n}(F),\\
2n & \text{ when } G_n=\SO_{2n+1}(F).
\end{cases}
$$ 
\end{enumerate}
We remark that the bound $|x_i|<\half{1}$ follows from the trivial bound of the Ramanujan conjecture of general linear groups.

Two local Arthur parameters are called equivalent if they are conjugate under $\widehat{G}_n(\BC)$. By abuse of notation, we do not distinguish $\psi$ with its equivalence class in this paper. We let $\Psi^{+}(G_n)$ be the equivalence class of local Arthur parameters, and $\Psi(G_n)$ the subset of $\Psi^+(G_n)$ consisting of local Arthur parameters $\psi$ whose restriction to $W_F$ is bounded. In other words, $\psi\in\Psi(G_n)$ if and only if $x_i=0$ for $i=1,\dots, r$ in the decomposition \eqref{eq decomp psi +}.

By the local Langlands correspondence for $\GL_{d_i}(F)$, the bounded irreducible representations $\phi_i$ of $W_F$ can be identified with an irreducible unitary supercuspidal representations $\rho_i$ of $\GL_{d_i}(F)$ (\cite{Hen00, HT01, Sch13}). Consequently, we often write
\begin{equation}\label{A-param decomp}
  \psi = \bigoplus_{\rho}\left(\bigoplus_{i\in I_\rho} \rho|\cdot|^{x_i} \otimes S_{a_i} \otimes S_{b_i}\right),  
\end{equation}
where the first sum runs over 
irreducible unitary supercuspidal representations $\rho$ of $\GL_n(F)$, $n \in \mathbb{Z}_{\geq 1}$. Occasionally, we also write $\rho|\cdot|^{x}\otimes S_a=\rho|\cdot|^x\otimes S_a \otimes S_1.$ 

Let $\psi$ be a local Arthur parameter as in \eqref{A-param decomp}, we say that $\psi$ is of \emph{good parity} if $\psi \in \Psi(G_n)$ (i.e., $x_i=0$ for all $i$) and every summand $\rho \otimes S_{a_i} \otimes S_{b_i}$ is self-dual and of the same type as $\psi.$ That is, $\rho$ is self-dual and
\begin{itemize}
    \item if $G_n=\Sp_{2n}(F)$ and $\rho$ is orthogonal (resp. symplectic), then $a_i+b_i$ is even (resp. odd);
    \item if $G_n=\SO_{2n+1}(F)$ and $\rho$ is orthogonal (resp. symplectic), then $a_i+b_i$ is odd, (resp. even).
\end{itemize}
We let $\Psi_{gp}(G_n)$ denote the subset of $\Psi(G_n)$ consisting of local Arthur parameters of good parity.

Let $\psi \in \Psi^{+}(G_n).$ From the decomposition \eqref{A-param decomp}, define a subrepresentation $\psi_{nu,>0}$ of $\psi$ by
\[ \psi_{nu,>0}:= \bigoplus_{\rho}\left(\bigoplus_{\substack{i\in I_\rho,\\ x_i>0}} \rho|\cdot|^{x_i} \otimes S_{a_i} \otimes S_{b_i}\right). \]
Since the image of $\psi$ is contained in  $\widehat{G}_n(\BC)$ and $\psi$ is self-dual, $\psi$ also contains $(\psi_{nu,>0})^{\vee}$. Then, define $\psi_{u} \in \Psi(G_m)$ for some $m\leq n$ by 
\begin{align}\label{eq def of psi_u}
    \psi= \psi_{nu,>0} \oplus \psi_u \oplus (\psi_{nu,>0})^{\vee}.
\end{align}
Equivalently, we have
\[ \psi_{u}:= \bigoplus_{\rho}\left(\bigoplus_{\substack{i\in I_\rho,\\ x_i=0}} \rho \otimes S_{a_i} \otimes S_{b_i}\right). \]

In \cite[Theorem 1.5.1]{Art13}, for a local Arthur parameter $\psi \in \Psi(G_n)$, Arthur constructed a finite multi-set $\Pi_\psi$ consisting of irreducible unitary representations of $G_n.$ We call $\Pi_\psi$ the \emph{local Arthur packet} of $\psi.$ M{\oe}glin gave another construction of $\Pi_\psi$ and showed that it is multiplicity-free (\cite{Moe11}). For $\psi \in \Psi^+(G_n)$, Arthur defined the local Arthur packet $\Pi_\psi$ (\cite[(1.5.1)]{Art13}), by
\begin{align}\label{eq def packet +}
    \Pi_{\psi}:= \{ \tau_{\psi_{nu,>0}} \rtimes \pi_u \ | \ \pi_{u} \in \Pi_{\psi_u}   \},
\end{align}
where $\tau_{\psi_{nu,>0}}$ is the following irreducible representation of a general linear group
$$
\tau_{\psi_{nu,>0}}=\bigtimes_\rho\bigtimes_{ \substack{i\in I_\rho,\\ x_i>0}}\begin{pmatrix}
\frac{a_i-b_i}{2}+x_i & \cdots & \frac{a_i+b_i}{2}-1+x_i \\
\vdots & \ddots & \vdots \\
\frac{-a_i-b_i}{2}+1+x_i & \cdots & \frac{b_i-a_i}{2}+x_i
\end{pmatrix}_{\rho}.
$$
Since $|x_i|<\half{1}$ in the decomposition \eqref{A-param decomp}, the parabolic induction in \eqref{eq def packet +} is always irreducible by \cite[Proposition 5.1]{Moe11b} (see also \cite[Theorem 9.3(6)]{Jan97}, \cite[Proposition 3.2(i)]{Tad09}).  We say that an irreducible representation $\pi$ of $G_n$ is \emph{of Arthur type} if $\pi\in\Pi_\psi$ for some local Arthur parameter $\psi \in \Psi^+(G_n)$.

Next, we further decompose $\psi_{u}$. Suppose $\rho \otimes S_a \otimes S_b$ is an irreducible summand of $\psi_u$ that is either not self-dual, or self-dual but not of the same type as $\psi$. Then $\psi$ must also contain another summand $(\rho\otimes S_a\otimes S_b)^{\vee}=\rho^{\vee} \otimes S_a \otimes S_b$. Therefore, we may choose a subrepresentation $\psi_{np}$ of $\psi_u$ such that
\begin{align}\label{eq decomp of psi_u}
     \psi_{u}= \psi_{np} \oplus \psi_{gp} \oplus \psi_{np}^{\vee},
\end{align}
where $\psi_{gp} $ is of good parity, and any irreducible summand of $\psi_{np}$ is either not self-dual or self-dual but not of the same type as $\psi$. In \cite{Moe06a}, M{\oe}glin constructed the local Arthur packet $\Pi_{\psi_u}$ from $\Pi_{\psi_{gp}}$, which we record below.

\begin{thm}[{\cite[Theorem 6]{Moe06a}, \cite[Proposition 8.11]{Xu17b}}]\label{thm reduction to gp}
Let $\psi_u \in \Psi(G_n)$ with a choice of decomposition \eqref{eq decomp of psi_u}. Write
\[ \psi_{np}= \bigoplus_{\rho} \left( \bigoplus_{i \in I_{\rho}} \rho \otimes S_{a_i} \otimes S_{b_i}\right),  \]
and consider the following irreducible parabolic induction
$$
\tau_{\psi_{np}}=\bigtimes_\rho\bigtimes_{i\in I_\rho}\begin{pmatrix}
\frac{a_i-b_i}{2} & \cdots & \frac{a_i+b_i}{2}-1 \\
\vdots & \ddots & \vdots \\
\frac{-a_i-b_i}{2}+1 & \cdots & \frac{b_i-a_i}{2}
\end{pmatrix}_{\rho}.
$$
Then for any $\pi_{gp}\in\Pi_{\psi_{gp}}$ the induced representation $\tau_{\psi_{np}}\rtimes\pi_{gp}$ is irreducible, independent of choice of $\psi_{np}$. Moreover,
$$
\Pi_\psi=\{\tau_{\psi_{np}}\rtimes\pi_{gp} \, | \, \pi\in\Pi_{\psi_{gp}}\}.
$$
\end{thm}

Combined with \eqref{eq def packet +}, we obtain the following.

\begin{thm}[{\cite[Proposition 5.1]{Moe11b}}]\label{thm red from nu to gp}
Let $\psi\in\Psi^+(G_n)$ with decomposition $\psi=\psi_{nu,>0}+\psi_{np}+\psi_{gp}+\psi_{np}^\vee+\psi_{nu,>0}^\vee$ as above. Then, for any $\pi_{gp}\in\Pi_{\psi_{gp}},$ the induction $\tau_{\psi_{nu,>0}}\times\tau_{\psi_{np}}\rtimes\pi_{gp}$ is irreducible. As a consequence, \begin{equation}\label{non-unitary A-packet}
    \Pi_\psi=\{\tau_{\psi_{nu,>0}}\times\tau_{\psi_{np}}\rtimes\pi_{gp} \ | \ \pi_{gp}\in\Pi_{\psi_{gp}}\}.
\end{equation}
\end{thm}

Theorem \ref{thm red from nu to gp} reduces calculating intersections of local Arthur packets to the good parity case.
We give an analogous definition of good parity for representations.

\begin{defn}\label{def good parity reps}
We say an irreducible representation
\[ \pi= L(\Delta_{\rho_1}[x_1,y_1],\dots,\Delta_{\rho_r}[x_r,y_r]; \pi_{temp})\] 
of $G_n$ is of \emph{good parity} if the following conditions hold:
\begin{enumerate}
    \item [$\oldbullet$] The tempered representation $\pi_{temp}$ lies in $\Pi_{\psi_{temp}}$ for some $\psi_{temp} \in \Psi_{gp}(G_m)$.
    \item [$\oldbullet$] For $1 \leq i \leq r$, $x_i,y_i \in \half{1} \Z$ and $\rho_i \otimes S_{x_i-y_i+1}\otimes S_1 $ is self-dual of the same type as $\widehat{G}_n$.
\end{enumerate}
\end{defn}

Let $\pi \in \Pi_{\psi}$. By the construction of local Arthur packets in the good parity case and Theorem \ref{thm reduction to gp}, we have that $\pi$ is of good parity if and only if $\psi$ is of good parity. Here is an immediate corollary.

\begin{cor}[{\cite[Corollary 2.13(1)]{HLL22}}]\label{cor reduction from nu to gp} For $\pi \in \Pi(G_n)$, there exists $\psi \in \Psi^+(G_n)$ such that $\pi \in \Pi_{\psi}$ if and only if $\pi$ is of the form
    \begin{align*}
        \pi= \tau_{\psi_{nu,>0}} \times \tau_{\psi_{np}} \rtimes \pi_{gp},
    \end{align*}
    and $\pi_{gp} \in \Pi_{\psi_{gp}}$ for some $\psi_{gp} \in \Psi_{gp}(G_n)$. Moreover, we have
    \begin{align*}
        \Psi(\pi):=& \{ \psi \in \Psi^+(G_n)\ | \ \pi \in \Pi_{\psi}\}\\
        =& \{ \psi_{nu,>0} + \psi_{np} + \psi_{gp} + \psi_{np}^{\vee}+ \psi_{nu,>0} \ | \ \psi_{gp} \in \Psi(\pi_{gp})\}.
    \end{align*}
\end{cor}

\subsection{Parametrization of tempered spectrum}\label{sec tempered parametrization}

Let $\psi \in \Psi(G_n)$. An important ingredient of the construction of local Arthur packets is a map
\begin{align}\label{eq character of Arthur packet}
    \Pi_{\psi} &\longrightarrow \widehat{\mathcal{S}}_{\psi},\\
   \nonumber \pi & \longmapsto \langle \cdot, \pi \rangle,
\end{align}
where $\widehat{\mathcal{S}}_{\psi}$ denotes Pontryagin dual of the \emph{component group}
\[  \mathcal{S}_{\psi}:= \pi_0( \mathrm{Cent}_{\widehat{G}_n(\BC)}(\mathrm{Im}(\psi))/Z(\widehat{G}_n(\BC))). \]
The map \eqref{eq character of Arthur packet} is neither injective nor surjective in general. However, when $\psi$ is tempered, it is a bijection (see Theorem \ref{thm Arthur tempered} below). In this subsection, we recall the combinatorial description of $\widehat{\mathcal{S}}_{\psi}$ in \cite[\S 2]{Xu17b} and use it to give a parametrization of the tempered spectrum of $G_n$.

Write $\psi=\psi_{np} + \psi_{gp}+ \psi_{np}^{\vee}$. There is a bijection between $ \mathcal{S}_{\psi}$ and $\mathcal{S}_{\psi_{gp}}$. Therefore, to describe $\widehat{\mathcal{S}}_{\psi}$, we may assume $\psi \in \Psi_{gp}(G_n)$. Write
\begin{align}\label{eq decomp psi parametrization of tempered}
    \psi= \bigoplus_{\rho} \bigoplus_{i \in I_{\rho}} \rho \otimes S_{a_i} \otimes S_{b_i}.
\end{align}
First, we consider the \emph{enhanced component group} of $\psi$ defined by
$$
\mathcal{A}_\psi=\bigoplus_{\rho}\bigoplus_{i\in I_\rho}(\mathbb{Z}/2\mathbb{Z}) \alpha_{\rho,i}.
$$
Here, $\mathcal{A}_\psi$ is the finite vector space over $\mathbb{Z}/2\mathbb{Z}$ with basis $\alpha_{\rho,i}$ corresponding to the summands $\rho \otimes S_{a_i} \otimes S_{b_i}$ of Equation (\ref{eq decomp psi parametrization of tempered}). It is possible that $\rho \otimes S_{a_i} \otimes S_{b_i}=\rho \otimes S_{a_j} \otimes S_{b_j}$ for some $i\neq j\in I_\rho$; however, we distinguish these summands in $\mathcal{A}_\psi,$ i.e.,  $\alpha_{\rho,i}\neq\alpha_{\rho,j}$ in $\mathcal{A}_\psi.$ The \emph{central element} of $\mathcal{A}_\psi$ is $z_\psi:=\sum_\rho\sum_{i\in I_\rho}\alpha_{\rho,i}.$

The component group $\mathcal{S}_\psi$ can be identified with the quotient of $\mathcal{A}_\psi$ by the subgroup generated by the central element and the elements $\alpha_{\rho,i}+\alpha_{\rho,j}$ such that $\rho \otimes S_{a_i} \otimes S_{b_i}=\rho \otimes S_{a_j} \otimes S_{b_j}.$ Hence, we identify $\widehat{\mathcal{S}}_{\psi}$ with the set of functions $\varepsilon$ from the irreducible summands of $\psi$ to $\{\pm 1\}$ that satisfy
\begin{enumerate}
    \item [$\oldbullet$] $\varepsilon(\rho \otimes S_{a_i} \otimes S_{b_i})=\varepsilon(\rho \otimes S_{a_j} \otimes S_{b_j})$ if $\rho \otimes S_{a_i} \otimes S_{b_i}=\rho \otimes S_{a_i} \otimes S_{b_i}$, and
    \item [$\oldbullet$] $\prod_{\rho} \prod_{i \in I_{\rho}} \varepsilon(\rho \otimes S_{a_i}\otimes S_{b_i})=1$.
\end{enumerate}

The following theorem is Arthur's classification of the tempered representations.

\begin{thm}[{\cite[Theorem 1.5.1]{Art13}}]\label{thm Arthur tempered}
Any irreducible tempered representation of $G_n$ lies in $\Pi_\psi$ for some tempered local Arthur parameter $\psi.$ Moreover, if $\psi_1 $ and $\psi_2$ are two non-equivalent tempered local Arthur parameters, then 
$$
\Pi_{\psi_1}\cap\Pi_{\psi_2}=\emptyset.
$$
Finally, if one fixes a choice of Whittaker datum for $G_n$ and $\psi$ is tempered, then there is a bijective map between the tempered local Arthur packet $\Pi_{\psi}$ and $\widehat{\mathcal{S}}_\psi.$
\end{thm}

Hereinafter, we fix a choice of Whittaker datum for $G_n.$ When $\psi$ is tempered and of good parity, we write $\pi(\psi,\varepsilon)$ for the representation in $\Pi_\psi$ corresponding to $\varepsilon\in\widehat{\mathcal{S}}_\psi$ via the bijection in Theorem \ref{thm Arthur tempered}.

\section{Atobe's reformulation}\label{sec atob refor}
In this section, we recall the main definitions and results of \cite{Ato20b} for the construction of local Arthur packets of good parity for $G_n$. We also recall the operators and the main results of \cite{HLL22}.

We fix the following notation throughout this section. Let $\psi$ be any local Arthur parameter of good parity of $G_n$ with decomposition 
\[ \psi= \bigoplus_{\rho} \bigoplus_{i \in I_{\rho}} \rho \otimes S_{a_i} \otimes S_{b_i}. \]
 We set $A_i=\frac{a_i+b_i}{2}-1$ and $B_i=\frac{a_i-b_i}{2}$ for $i\in I_\rho.$ 
 
A total order $>_\psi$ on $I_\rho$ is called \emph{admissible} if satisfies:
\[
\tag{$P$}
\text{
For $i,j \in I_\rho$, 
if $A_i > A_j$ and $B_i > B_j$, 
then $i >_\psi j$.
}
\]
Sometimes we consider an order $>_\psi$ on $I_\rho$ satisfying:
\[
\tag{$P'$}
\text{
For $i,j \in I_\rho$, 
if $B_i > B_j$, 
then $i >_\psi j$.
}
\]
Note that ($P'$) implies ($P$). Often, we write $>$ instead of $>_\psi$ when the admissible order is fixed. 

Suppose an admissible order for $\psi$ is fixed. We define the collection of ordered multi-sets 
$$\supp(\psi) := \cup_{\rho}\{ [A_i,B_i]_{\rho} \}_{i \in (I_\rho,>)},
$$
called the support of $\psi$. Note that $\supp(\psi)$ depends implicitly on the fixed admissible order.

\subsection{Extended multi-segments and associated representations}

In this subsection, we recall Atobe's parametrization of local Arthur packets from \cite{Ato20b}.

\begin{defn} [{\cite[Definition 3.1]{Ato20a}}]
(Extended multi-segments)\label{def multi-segment}
\begin{enumerate}
\item
An \emph{extended segment} is a triple $([A,B]_\rho, l, \eta)$,
where
\begin{itemize}
\item
$[A,B]_\rho = \{\rho|\cdot|^A, \rho|\cdot|^{A-1}, \dots, \rho|\cdot|^B \}$ is a segment 
for an irreducible unitary supercuspidal representation $\rho$ of some $\GL_d(F)$; 
\item
$l \in \Z$ with $0 \leq l \leq \frac{b}{2}$, where $b = \#[A,B]_\rho = A-B+1$; 
\item
$\eta \in \{\pm1\}$. 
\end{itemize}

\item
An \emph{extended multi-segment} for $G_n$ is 
an equivalence class (via the equivalence defined below) of multi-sets of extended segments 
\[
\EE = \cup_{\rho}\{ ([A_i,B_i]_{\rho}, l_i, \eta_i) \}_{i \in (I_\rho,>)}
\]
such that 
\begin{itemize}
\item
$I_\rho$ is a totally ordered finite set with a fixed admissible total order $>$;

\item
$A_i + B_i \geq 0$ for all $\rho$ and $i \in I_\rho$; 

\item
as a representation of $W_F \times \SL_2(\BC) \times \SL_2(\BC)$, 
\[
\psi_{\EE} = \bigoplus_\rho \bigoplus_{i \in I_\rho} \rho \otimes S_{a_i} \otimes S_{b_i} 
\]
where $(a_i, b_i) = (A_i+B_i+1, A_i-B_i+1)$,
is a local Arthur parameter for $G_n$ of good parity. We shall denote $\psi_{\EE}$ the local Arthur parameter associated with $\EE$. 
\item The sign condition
\begin{align}\label{eq sign condition}
\prod_{\rho} \prod_{i \in I_\rho} (-1)^{[\frac{b_i}{2}]+l_i} \eta_i^{b_i} = 1
\end{align}
holds.
\end{itemize}

\item
Two extended segments $([A,B]_\rho, l, \eta)$ and $([A',B']_{\rho'}, l', \eta')$ are \emph{weakly equivalent} 
if 
\begin{itemize}
\item
$[A,B]_\rho = [A',B']_{\rho'}$; 
\item
$l = l'$; and 
\item
$\eta = \eta'$ whenever $l = l' < \frac{b}{2}$. 
\end{itemize}
Two extended multi-segments 
$\EE = \cup_{\rho}\{ ([A_i,B_i]_{\rho}, l_i, \eta_i) \}_{i \in (I_\rho,>)}$ 
and 
$\EE' = \cup_{\rho}\{ ([A'_i,B'_i]_{\rho}, l'_i, \eta'_i) \}_{i \in (I_\rho,>)}$ 
are \emph{weakly equivalent}
if for any $\rho$ and $i \in I_\rho$, the extended segments $([A_i,B_i]_\rho, l_i, \eta_i)$ and $([A'_i,B'_i]_{\rho}, l'_i, \eta'_i)$ are weakly equivalent.

\item
We define the \emph{support} of $\EE$ to be the collection of ordered multi-sets 
\[
\supp(\EE) = \cup_{\rho}\{ [A_i,B_i]_{\rho} \}_{i \in (I_\rho,>)}.
\]
\end{enumerate}
\end{defn}

If the admissible order $>$ is clear in the context, for $k \in I_{\rho}$, we often denote $k+1 \in I_{\rho}$ to be the unique element adjacent with $k$ and $k+1>k$.

For each extended multi-segment $\EE$, Atobe associated a representation $\pi(\EE)$, which is irreducible or zero, see \cite[\S3.2]{Ato20b}. The $L$-data of $\pi(\EE)$ can be computed explicitly by the algorithms in \cite{AM20}. In this paper, we only need the following special case, where we have a formula for $\pi(\EE)$. 
\begin{thm}[{\cite[Theorem 9.5]{HLL22}}] \label{thm (L)}
  We say an extended multi-segment $\EE= \cup_{\rho}([A_i,B_i]_{\rho},l_i,\eta_i)\}_{i \in (I_{\rho},>)}$ satisfies $(L)$ if the following conditions hold.
    \begin{enumerate}
        \item [$\oldbullet$] If $i < j \in I_{\rho}$, then $A_i+B_i \leq A_j+B_j$.
        \item [$\oldbullet$] For any $i \in I_{\rho}$, $l_i= \lfloor \half{A_i-B_i+1}\rfloor$.
        \item [$\oldbullet$] If $A_i+B_i=A_j+B_j$ and $A_i-B_i$ is even, then $\eta_i=\eta_j$.
    \end{enumerate}
   If $\EE$ satisfies $(L)$, then $\pi(\EE)\neq 0$. Moreover, we have
    \[ \pi(\EE)= L(\Delta_{\rho_1}[x_1,y_1],\ldots, \Delta_{\rho_f}[x_f,y_f]; \pi(\phi,\varepsilon)),\]
    where
    \begin{enumerate}
        \item [$\oldbullet$] as a multi-set, 
        \[\{[x_i,y_i]_{\rho_i}\}_{i=1,\ldots, f}= \sum_{\rho} \sum_{i \in I_{\rho}} \{ [B_i,-A_i]_{\rho}, [B_i+1,-A_i+1]_{\rho},\ldots, [B_i+l_i, -A_i+l_i]_{\rho}   \}; \]
        \item [$\oldbullet$] the tempered $L$-parameter $\phi$ is given by
        \[ \phi= \bigoplus_{\rho} \bigoplus_{ \substack{i \in I_{\rho},\\{2 | (A_i-B_i)}}} \rho \otimes S_{ A_i+B_i+1 },\]
        and $\varepsilon(\rho\otimes S_{ A_i+B_i+1 })= \eta_i$ for any $i \in I_{\rho}$ such that $A_i-B_i$ is even.
    \end{enumerate}
    In particular, $\Pi_{\phi_{\psi}}= \{\pi(\EE)\ | \ \psi_{\EE}= \psi \text{ and }\EE \text{ satisfies (L).}\}$
\end{thm}

Now we recall Atobe's parametrization of the local Arthur packets associated with local Arthur parameters of good parity.

\begin{thm}[{\cite[Theorem 3.3]{Ato20b}}]\label{thm Atobe's reformulation}
Suppose $\psi= \bigoplus_{\rho} \bigoplus_{i \in I_{\rho}} \rho \otimes S_{a_i} \otimes S_{b_i}$ is a local Arthur parameter of good parity of $G_n$. Choose an admissible order $>_{\psi}$ on $I_{\rho}$ for each $\rho$ that satisfies ($P'$) if $\half{a_i-b_i}<0$ for some $i \in I_{\rho}$. Then
\[ \bigoplus_{\pi \in \Pi_{\psi}} \pi= \bigoplus_{\EE} \pi(\EE),\]
where $\EE$ runs over all extended multi-segments with $\supp(\EE)= \supp(\psi)$ and $\pi(\EE) \neq 0$. 
\end{thm}

By $\supp(\EE)= \supp(\psi),$ we mean that they are equal as multi-sets and also that the admissible orders on $I_\rho$ agree.

Often, we need to study an extended multi-segment $\EE$ piece by piece. 
Suppose 
$$\EE= \cup_{\rho} \{([A_i,B_i]_{\rho}, l_i, \eta_i)\}_{i \in (I_{\rho},>)}$$
satisfies $\pi(\EE)\neq 0$.  We set  
\[ \EE_{\rho}=\{([A_i,B_i]_{\rho}, l_i, \eta_i)\}_{i \in (I_{\rho},>)},\  \EE^{\rho}=\cup_{\rho' \not\cong \rho}\{([A_i,B_i]_{\rho'}, l_i, \eta_i)\}_{i \in (I_{\rho'},>)}. \]
 For $i \in I_{\rho}$, we usually call the extended segment $ ([A_i,B_i]_{\rho}, l_i, \eta_i)$ the $i$-th row of $ \EE_{\rho}$. Finally, we define $\Rep$ to be the set of extended multi-segments $\EE$ such that $\pi(\EE) \neq 0$.

\subsection{\texorpdfstring{Intersection of local Arthur packets}{}}
In this subsection, we recall the main results in \cite{HLL22} and \cite{Ato23} on the intersection of local Arthur packets. We state them with the notation in \cite{HLL22}.

The main ingredients of these results are operators $R_k$, $ui_{i,j}$, $dual,$ $dual_i^{\pm}$ on extended multi-segments. In this paper, we only need certain consequences of the definition of $ui_{i,j}$ and $dual_i^{\pm}$, which is stated in Lemma \ref{lem operator}. We refer the explicit definitions of $ui_{i,j}, dual_i^{\pm}$ to \cite[Definition 3.23, 5.1, 6.5]{HLL22}. On the other hand, we recall the definition $dual$ and its compatibility with Aubert-Zelevinsky dual now.

Let $\widehat{\pi}$ denote the Aubert-Zelevinsky dual of $\pi$ defined in \cite{Aub95}. The Aubert-Zelevinsky dual is compatible with local Arthur packets in the following sense. Let $\psi=\bigoplus_\rho \bigoplus_{i \in I_\rho} \rho \otimes S_{a_i} \otimes S_{b_i}$ be a local Arthur parameter of good parity and set $\widehat{\psi}:=\bigoplus_\rho \bigoplus_{i \in I_\rho} \rho \otimes S_{b_i} \otimes S_{a_i}.$ Note that $\widehat{\psi}$ is also a local Arthur parameter of good parity. By the results in \cite[\S A]{Xu17b}, we have 
$$
\Pi_{\widehat{\psi}}=\{\widehat{\pi} \, | \, \pi\in\Pi_\psi\}.
$$
Therefore, given $\EE \in \Rep$, it is natural to ask what is the formula for the extended multi-segment $dual(\EE)$ such that $\pi(dual(\EE)) \cong \widehat{\pi(\EE)}$ and $\psi_{\EE'}= \widehat{\psi}_{\EE}$. This is answered in \cite[\S6]{Ato20b}. 
\begin{defn}[{\cite[Definition 6.1]{Ato20b}}]\label{dual segment}
Let $\EE= \cup_\rho \{([A_i,B_i]_{\rho},l_i,\eta_i)\}_{i\in (I_\rho, >)}$ be an extended multi-segment such that the admissible order $>$ on $I_{\rho}$ satisfies (P') for all $\rho$.  Define 
\[dual(\EE)=\cup_{\rho}\{([A_i,-B_i]_{\rho},l_i',\eta_i')\}_{i\in (I_\rho, >')}\] as follows:
\begin{enumerate}
    \item The order $>'$ is defined by $i>'j$ if and only if $j>i.$ 
    \item We set \begin{align*}
l_i'=\begin{cases}
l_i+B_i  & \mathrm{if} \, B_i\in\mathbb{Z},\\
 l_i+B_i+\frac{1}{2}(-1)^{\alpha_{i}}\eta_i  & \mathrm{if} \, B_i\not\in\mathbb{Z},
\end{cases}
\end{align*}
and
\begin{align*}
\eta_i'=\begin{cases}
(-1)^{\alpha_i+\beta_i}\eta_i  & \mathrm{if} \, B_i\in\mathbb{Z},\\
 (-1)^{\alpha_i+\beta_i+1}\eta_i  & \mathrm{if} \, B_i\not\in\mathbb{Z},
\end{cases}
\end{align*}
where $\alpha_{i}=\sum_{j\in I_\rho, j<i}a_j,$ and $\beta_{i}=\sum_{j\in I_\rho, j>i}b_j,$ $a_j=A_j+B_j+1$, $b_j=A_j-B_j+1$.
\item When $B_i\not\in\mathbb{Z}$ and $l_i=\frac{b_i}{2}$, we set $\eta_i=(-1)^{\alpha_i+1}.$
\end{enumerate}
If $\FF= \EE_{\rho}$, define $dual(\FF):= (dual(\EE))_{\rho}$.
\end{defn}

Now we state the main theorems of \cite{HLL22}. 

\begin{thm}[{\cite[Theorem 1.4]{HLL22}}]\label{main HLL24}
Let $\EE_1$ and $\EE_2$ be two extended multi-segments for $G_n$. Suppose that $\pi(\EE_1) \in \Pi_{\psi}$. 
Then the following holds. 
\begin{enumerate}
    \item Let $T$ be any of the four operators $R_k$, $ui_{i,j}$, $dual \circ ui_{i,j} \circ dual$, $dual_k$ or their inverses. We have 
    $T(\EE_1)$ is also an extended multi-segment for $G_n$, and 
    $$\pi(T(\EE_1))=\pi(\EE_1).$$
    \item  We have $\pi(\EE_1) \cong \pi(\EE_2)$ if and only if 
$\EE_2$ can be obtained from $\EE_1$
by a finite chain of the four operators and their inverses in Part (1). 
\item There is a precise formula/algorithm to compute the set $\{ \psi' \ | \ \pi(\EE_1) \in \Pi_{\psi'}\}$.
\end{enumerate}
\end{thm}

We remark that the operator $R_k$ does not change the corresponding local Arthur parameter, i.e. $\psi_{\EE}= \psi_{R_k(\EE)}$. Thus, we often omit $R_k$ in the computation for $\Psi(\pi)$.

\section{The enhanced Shahidi conjecture}\label{section enhanced Shahidi conjecture}
In this section, we prove the enhanced Shahidi conjecture for symplectic and split odd special orthogonal groups $G_n$ (see Theorems \ref{proof of enh sha conj}, \ref{thm non-unitary Enhanced Shahidi}). 

\subsection{\texorpdfstring{A criterion for $|\Psi(\pi)|=1$ for tempered representation}{}}
In this subsection, we give a criterion for $|\Psi(\pi)|=1$ for any tempered representation $\pi=\pi(\phi,\varepsilon)$.

First, we state several immediate consequences of the definition of the operators $ui_{i,j}, dual_k^{+}, dual_k^{-}$ (see \cite[Definitions 3.23, 5.1, 6.5]{HLL22}). Here we recall that two segments $\Delta_1, \Delta_2 $ are not linked if $ \Delta_{1}\cup \Delta_2$ is not a segment or either $\Delta_1 \supseteq \Delta_2$ or $\Delta_1 \subseteq \Delta_2$.

\begin{lemma}\label{lem operator}
Let $\EE= \cup_{\rho} \{ ([A_i,B_i]_{\rho},l_i,\eta_i)  \}_{i \in (I_{\rho},>)} \in \Rep$. Let $i,j \in I_\rho$ and $i+1>i$ be adjacent.
\begin{enumerate}
    \item [(a)]If $[A_i,B_i]_{\rho} $ and $ [A_j,B_j]_{\rho}$ are not linked, then $ui_{i,j}$ is not applicable on $\EE$.
    \item [(b)] Suppose $A_i=B_i$ and $A_{i+1}=B_{i+1}$. Then, the operator $ui_{i,i+1}$ is applicable on $\EE$ if and only if $A_{i+1}=A_i+1$ and $\eta_i=-\eta_{i+1}$.
    \item [(c)] Suppose $i$ is minimal with respect to $>$ and $B_i=1/2$. Then, $dual_i^{+}$ is applicable on $\EE$ if and only if $\eta_i=-1$.
    \item [(d)] The operator $dual_i^{-}$ is applicable on $\EE$ only if $B_i=-1/2$.
\end{enumerate}
\end{lemma}
\begin{proof}
    All of these follow directly from the definitions of these operators, which we omit the details.
\end{proof}

Now, we describe when a tempered representation of good parity lies in only one local Arthur packet. Note that these conditions are exactly the opposite of the conditions for $\pi(\phi,\varepsilon)$ being supercuspidal (see \cite[Theorem 2.5.1]{Moe11} or \cite[Theorem 3.3]{Xu17a}).

\begin{thm}\label{thm tempered one A-packet}
Let $\pi=\pi(\phi,\varepsilon)$ be a tempered representation of $G_n$ of good parity. Then $\Psi(\pi):= \{ \psi \ | \  \pi \in \Pi_{\psi}\}$ is a singleton if and only if the following conditions hold.
\begin{enumerate}
    \item [$\oldbullet$] If $ \rho \otimes S_{a} \subset \phi$ and $\rho \otimes S_{a+2} \subset \phi$, then $\varepsilon(\rho \otimes S_{a} )\varepsilon(\rho \otimes S_{a+2} )=1$.
    \item [$\oldbullet$] If $ \rho \otimes S_{2} \subset \phi$, then $\varepsilon(\rho \otimes S_{2} )=1$.
\end{enumerate}
\end{thm}
\begin{proof}
Write
\[ \phi= \bigoplus_{\rho} \bigoplus_{i \in I_{\rho}} \rho \otimes S_{2 z_i +1} , \]
 and let $\EE= \cup_{\rho}\{ ([z_i,z_i]_{\rho},0, \varepsilon(\rho \otimes S_{2z_i+1}))\}_{i \in (I_{\rho,>})}$ so that $\pi(\EE)= \pi$ (see Theorem \ref{thm (L)}). Here the total order $>$ is non-decreasing with respect to $z_i$.
 
 First, we claim that if any of the conditions fails, then there is an operator $T$ applicable on $\EE$, and hence $\Psi(\pi) \supseteq \{\psi_{\EE} \neq \psi_{T(\EE)}\}$ is not a singleton by Theorem \ref{main HLL24}(1). Indeed, if the first condition fails, say $z_{i+1}=z_i+1$ and $\varepsilon(\rho\otimes S_{2z_i+1})=-\varepsilon(\rho\otimes S_{2z_{i}+3})$, then $T=ui_{i,i+1}$ is applicable on $\EE$ by Lemma \ref{lem operator}(b). If the second condition fails, let $i$ be the minimal element in $(I_{\rho},>)$. Then we must have $A_i=B_i=1/2$ and $\eta_i=-1$. Thus, $T=dual_i^{+}$ is applicable on $\EE$ by Lemma \ref{lem operator}(c). This completes the verification of the claim.

Next, we show that if the conditions hold, then there is no operator applicable on $\EE$, and hence $\Psi(\pi)$ is a singleton by Theorem \ref{main HLL24}(2). We verify that $ui_{i,j}, dual_i^{+}, dual \circ ui_{j,i} \circ dual, dual_i^{-}$ are not applicable case by case for any $i<j \in I_{\rho}$.

Suppose the contrary that $ui_{i,j}$ is applicable on $\EE$. Lemma \ref{lem operator}(a) implies that $[z_i,z_i]_{\rho}$ and $[z_{j},z_j]_{\rho}$ are linked, which holds only if $z_j=z_i+1$. However, the first condition in the statement contradicts to the applicability of $ui_{i,j}$ by Lemma \ref{lem operator}(b). Similarly, suppose the contrary that $dual_i^{+}$ is applicable on $\EE$. Then $z_i=1/2$ and hence we can assume $i$ is minimal in $(I_{\rho},>)$. Again, the second condition in the statement contradicts to the applicability of $dual_i^{+}$ by Lemma \ref{lem operator}(c).

To show $dual \circ ui_{j,i} \circ dual$ is never applicable on $\EE$, observe that the corresponding segments in $dual(\EE)$ are $[z_i,-z_i]$ and $[z_j,-z_j]$, which are never linked. Thus, $ui_{j,i}$ is never applicable on $dual(\EE)$ by Lemma \ref{lem operator}(a), and hence $dual \circ ui_{j,i} \circ dual$ is never applicable on $\EE$. Finally, none of the $z_i$ is equal to $-1/2$, and hence $dual_i^{-}$ is never applicable on $\EE$ by Lemma \ref{lem operator}(d). This completes the proof of the theorem.
 \end{proof}

\subsection{The enhanced Shahidi conjecture} In this subsection, we give a complete proof of the enhanced Shahidi conjecture for $G_n$.
Note that we have implicitly fixed a choice of Whittaker datum. That is, we fixed a $G_n$-conjugacy class of a tuple $(B,\chi)$, where $B$ is an $F$-rational Borel subgroup of $\mathrm{G}_n$ and $\chi$ is a generic character of the $F$-points of the unipotent radical of $B$. For a tempered local Arthur parameter $\psi$, this choice fixes the bijection between $\Pi_\psi$ and $\widehat{\mathcal{S}}_\psi$ in Theorem \ref{thm Arthur tempered}

In \cite[Conjecture 9.4]{Sha90}, Shahidi conjectured that for any Whittaker data $(B',\chi'),$ a tempered $L$-packet contains a $(B',\chi')$-generic representation. For Archimedean fields, Shelstad established a stronger form of this conjecture (\cite{She08}). Namely, for any Whittaker data $(B,\chi),$ a tempered $L$-packet contains a unique $(B,\chi)$-generic representation. Moreover, if the bijection between $\Pi_\psi$ and $\widehat{\mathcal{S}}_\psi$ is fixed by $(B,\chi),$ then this representation corresponds to the trivial character in $\widehat{\mathcal{S}}_\psi$. The non-Archimedean analogue of Shelstad's result is also true.
The Shahidi conjecture is known for symplectic and split odd orthogonal groups $G_n$ as follows (see {\cite{Kon02, JS03, Liu11, JS12, MW12, Wal12, Art13, Var17}}). 

\begin{thm}\label{thm Arthur generic}
Suppose $\phi$ is a tempered L-parameter of $G_n$, then $\pi(\phi,\varepsilon)$ is generic with respect to the Whittaker datum $(B,\chi)$ if and only if the character $\varepsilon$ is trivial.
\end{thm}

The enhanced Shahidi conjecture has global applications. If an automorphic representation $\pi$ lies in a global Arthur packet and $\pi_{v}$ is generic for some finite place $v$, then we can obtain information about the global Arthur parameter. More details will be given in \S \ref{sec global} (see 
Propositions \ref{thm partial Clozel 4} and \ref{thm stronger condition}). We remark that certain cases of the global application are already elaborated in the literature. 
To be explicit, Magaard and Savin (\cite[Proposition 8.2]{MS20}, $\mathrm{G}=\Sp_{2n}$) and Chen (\cite[Lemma 4.5]{Che23}, $\mathrm{G}=\Sp_{2n}$, $\SO_{2n+1}$) proved that if a local component of a cuspidal automorphic representation is 
a Steinberg representation, then its functorial lifting to $\GL$ is a cuspidal representation. Indeed, the key point is that Steinberg representations only live in one local Arthur packet, whose parameter is tempered and irreducible. See \cite[Proposition 8.2]{MS20} and Remark \ref{rmk Steinberg} for more details.

Conjecture \ref{Enhanced Shahidi conjecture intro} requires that we prove it for any representation that is generic with respect to some Whittaker datum; however, it is enough to prove it for a fixed Whittaker datum. Indeed, for a tempered local Arthur parameter $\psi,$ the bijection between $\Pi_\psi$ and $\widehat{\mathcal{S}}_\psi$ in Theorem \ref{thm Arthur tempered}, relies on the choice of Whittaker datum. If we let $(B',\chi')$ be another choice of Whittaker datum, then the bijection between the tempered local Arthur packet $\Pi_\psi$ and $\widehat{\mathcal{S}}_\psi$  may change; however, Theorem \ref{thm Arthur generic} shows that $\Pi_\psi$ still has a generic member with respect to the Whittaker datum $(B',\chi').$

Here is the strategy of the proof of the enhanced Shahidi conjecture for $G_n.$ First, irreducible generic representations have been classified for $G_n$ (Theorem \ref{thm generic sequence}). If $\pi$ is of good parity and generic, this allows us to show that $\pi$ must be tempered (Theorem \ref{proof of enh sha conj}(i)). By Theorem \ref{thm Arthur generic}, we obtain an extended multi-segment $\EE$ for which $\pi(\EE)=\pi$ and $\psi_\EE$ is tempered. 
We check that $\EE$ satisfies the conditions of Theorem \ref{thm tempered one A-packet} and hence $\pi$ only lies in the tempered local Arthur packet $\Pi_{\psi_\EE}$. 

We proceed with the details now. The following is the classification of generic representations of $G_n$ in terms of their $L$-data.

\begin{thm}[{\cite[\S 4]{Mui98}, \cite[Definition 5.1]{JS04}, \cite[Definition 4.14]{Liu11}}]\label{thm generic sequence}
Any irreducible generic representation $\sigma$ of $\SO_{2n+1}(F)$ (resp. $\Sp_{2n}(F)$) is an irreducible parabolic induction of the form (this gives the $L$-data of $\sigma$)
\[ \sigma = \Delta_{\rho_1}[x_1,-y_1]\times\cdots\times \Delta_{\rho_f}[x_f,-y_f] \rtimes \pi(\phi,1),\]
where the sequence of segments $\{ \Sigma_i:=[y_i,-x_i]_{\rho_i} \}_{i=1}^f$ is $\SO_{2n+1}$-generic (resp. $\Sp_{2n}$-generic) with respect to the tempered generic representation $\pi(\phi,1)$.

\end{thm}

We do not recall the explicit description of $\SO_{2n+1}$-generic (resp. $\Sp_{2n}$-generic) here. We shall only use the following special case of the definition. Suppose $\{\Sigma_i\}_{i=1}^f$ is $\SO_{2n+1}$-generic (resp. $\Sp_{2n}$-generic). Then, if $\rho_i$ is self-dual of orthogonal type (resp. symplectic type), then $\Sigma_i \neq [1/2,1/2]_{\rho_i}$. In particular, if $\sigma$ is generic, of good parity, and lies in the local Arthur packet $\Pi_{\psi}$, then there is no segment of the form $\Delta_{\rho}[-1/2,-1/2]$ in the $L$-data of $\sigma$.

Here are two observations.
\begin{lemma} \label{lem gen observation} \
\begin{enumerate}
    \item [(a)] Suppose $\phi= \sum_{i=1}^n \rho_i \otimes S_{2z_i+1}$ is a tempered L-parameter of $G_n$ of good parity. For any $z > 1/2$, the order of the highest derivative of $\pi(\phi,1)$ with respect to $D_{\rho|\cdot|^{z}}$ is given by
    \[ k = \#\{ i \ | \ \rho_i \cong \rho , z_i=z \} \]
    and $D_{\rho|\cdot|^{z}}^{(k)}(\pi(\phi,1))= \pi(\phi',1)$, where 
    \[ \phi'= \phi-  (\rho\otimes S_{2z+1})^{\oplus k} + (\rho\otimes S_{2(z-1)+1})^{\oplus k}.\]
    \item [(b)] Suppose $\EE\in\Rep$ and $\EE_{\rho}$ is of the form
    \[\EE_\rho=\{ ([A_k,B_k]_{\rho},l_k,\eta_k)\}_{k=1 }^{n-r}+ \{([y,y]_{\rho},0,\eta_n)^{r}\},\]
    where $A_k <y $ for all $1 \leq k \leq n-r$. If 
    \[ \pi(\EE)= L(\Delta_{\rho_1}[x_1,-y_1], \dots, \Delta_{\rho_f}[x_f,-y_f];\pi( \phi,\varepsilon) ),\]
    then  $x_i, y_i <y$ for $\rho_i \cong \rho$, and $\phi$ contains $r$ copies of $\rho\otimes S_{2y+1}$.
\end{enumerate}
\end{lemma}
\begin{proof}
Part (a) follows directly from \cite[Theorem 7.1]{AM20}. For Part (b), it is not hard to see from the definition of $\pi(\EE)$ in \cite[\S 3/2]{Ato20b} that the conclusion holds when $y$ is sufficiently large. The general case follows from taking derivatives (in the sense of \cite[Lemma 4.3(iv)]{HLL22}) which changes $\phi$ to $\phi- (\rho\otimes S_{2y+1})^{\oplus r}+(\rho\otimes S_{2(y-1)+1})^{\oplus r}$ at each stage. This completes the proof of the lemma.  
\end{proof}

Now we are ready to prove Conjecture \ref{Enhanced Shahidi conjecture intro}(1). 

\begin{thm}\label{proof of enh sha conj} 
Let $G_n$ be a symplectic or split odd special orthogonal group.
\begin{enumerate}
    \item [(i)] Suppose $\sigma \in \Pi_{\psi}$ for some $\psi \in \Psi(G_n)$. Then $\sigma$ is generic only if $\sigma$ is tempered. Hence, combining with Theorem \ref{thm Arthur generic}, Conjecture \ref{Enhanced Shahidi conjecture intro}(1) holds for $G_n$. 
    \item [(ii)] Any tempered generic representation lives in exactly one local Arthur packet.
\end{enumerate}
\end{thm}
\begin{proof}
For Part (i), we first deal with the case that $\psi$ is of good parity. In this case, we write $\sigma =\pi(\EE)$ for $\EE=\cup_{\rho} \EE_{\rho} \in \Rep$. 

Suppose for contradiction that $\sigma$ is not tempered. By Theorem \ref{thm generic sequence}, we write  
\begin{align}\label{eq parabolic induction of generic}
    \sigma= \Delta_{\rho_1}[x_1,-y_1]\times\cdots\times \Delta_{\rho_f}[x_f,-y_f] \rtimes \pi(\phi,1),
\end{align}
and 
\[ \phi= \sum_{j=f+1}^m \rho_j\otimes S_{2z_j +1} .\]
Let $\rho:=\rho_1$ and denote 
\begin{align*}
    y&= \max\{y_i\ | \ \rho_i \cong \rho\},\\
    z&= \max\left(\{ z_j \ | \ \rho_j \cong \rho \} \cup \{0\}\right).
\end{align*}
Note that for any $1 \leq i \leq m$, $ y_i> x_i \geq -y_i$ and Theorem \ref{thm generic sequence} implies $y_i > 1/2$. Also, $y\neq 0$. Now we consider two cases.

\textbf{Case 1.} Suppose $z \leq y$. Write $\EE_{\rho}= \{ ([A_k,B_k]_{\rho},l_k,\eta_k)\}_{k=1}^n$. In this case, we have
    \[y= \max\{ \alpha \ | \ \rho|\cdot|^{\alpha} \in \Omega(\pi(\EE))_{\rho} \}.\] Therefore, \cite[Theorem 4.10]{HLL22} implies
    \[ \# \{ i \ | \ \rho_i \cong \rho,\ y_i= y\}+ \# \{ j \ | \  \rho_j \cong \rho,\ z_j=y\} = \# \{ k \ | \ A_k = y \}=\# \{ k \ | \ A_k \geq y \}. \]
    On the other hand, from the form of (\ref{eq parabolic induction of generic}), the quantity 
    \[\# \{ i \ | \ \rho_i \cong \rho,\  y_i= y\}+ \# \{ j \ | \ \rho_j \cong \rho,\  z_j=y\}\]
    is precisely the order of the highest derivative of $\sigma$ with respect to $D_{\rho|\cdot|^{y}}$ by Lemma \ref{lem Leibniz rule} and Lemma \ref{lem gen observation}(a). Therefore, \cite[Proposition 8.3]{Xu17b} implies
    \[  \# \{ i \ | \  \rho_i \cong \rho,\  y_i= y\}+ \# \{ j \ | \  \rho_j \cong \rho,\ z_j=y\} \leq  \# \{ k \ | \ B_k = y \}. \]
   Then the inequality (from the condition $B_k \leq A_k$) 
   \[ \# \{ k \ | \ A_k \geq y \} \geq \# \{ k \ | \ B_k = y \}  \]
   implies that $\EE_{\rho}$ is of the form 
   \[ \EE_{\rho}= \{([A_k,B_k]_{\rho},l_k,\eta_k)\}_{k=1}^{n- r} + \{([y,y]_{\rho},0,1)^{r}\}  \]
   with $A_k < y$ for $1 \leq k \leq n-r$. However, this contradicts to Lemma \ref{lem gen observation}(b).
   
\textbf{Case 2.} Suppose $z>y$.
In this case, we construct a sequence of irreducible generic representations $\{ \sigma_z, \dots, \sigma_{y+1}, \sigma_y\}$ inductively. For $z \geq x \geq y$ such that $y-x \in \Z$, we define tempered local Arthur parameters $\phi_i$ inductively by
   \begin{enumerate}
       \item [$\oldbullet$] $ \phi_z= \phi$,
       \item [$\oldbullet$] $\pi(\phi_{x-1},1)$ is the highest derivative of $\pi(\phi_x,1)$ with respect to $\rho|\cdot|^{x}$. 
   \end{enumerate}
   Note that by Lemma \ref{lem gen observation}(a), we have 
   \[ \max\{ \alpha \ | \ \rho \otimes S_{2 \alpha+1} \subset \phi_{x}\} =x.\]
   Now we consider
   \[ \sigma_x:=\Delta_{\rho_1}[x_1,-y_1]\times\cdots\times \Delta_{\rho_f}[x_f,-y_f] \rtimes \pi(\phi_x,1). \]
   We claim that if $\sigma_x$ is irreducible and generic then so is $\sigma_{x-1}$.
   Indeed, we have $D_{\rho|\cdot|^{x}}^{(k_x)}( \pi( \phi_x,1))= \pi(\phi_{x-1},1)$, so it follows from Lemma \ref{lem Leibniz rule} that $D_{\rho|\cdot|^{x}}^{(k_x)}(\sigma_x)= \sigma_{x-1}$, and it is the highest derivative. In particular, $\sigma_{x-1}$ is an irreducible parabolic induction of a product of generic representations, and hence $\sigma_{x-1}$ is also generic, which verifies the claim. Now $\sigma_y$ is a generic irreducible representation in Case 1, which gives a contradiction.
   This completes the proof of Part (i) in the good parity case.

Next, we deal with the general case. By Theorem \ref{thm reduction to gp}, we decompose the local Arthur parameter $\psi= \psi_{np} \oplus \psi_{gp} \oplus \psi_{np}^{\vee}$, where $\psi_{gp}$ is of good parity. Then 
\[ \Pi_{\psi}= \{ \tau_{\psi_{np}} \rtimes \pi_{gp} \ | \ \pi_{gp} \in \Pi_{\psi_{gp}} \} \]
where if $\psi_{np}= \oplus_{\rho} (\oplus_{i \in I_{\rho}} \rho \otimes S_{a_i} \otimes S_{b_i})$, 
\[ \tau_{\psi_{np}}= \bigtimes_{\rho} \bigtimes_{i \in I_{\rho}} \begin{pmatrix}\frac{a_i-b_i}{2} &\cdots& \frac{a_i+b_i}{2}-1\\ \vdots & \ddots & \vdots\\ -\frac{a_i+b_i}{2}+1& \cdots & -\frac{a_i-b_i}{2}
\end{pmatrix}_{\rho}.\]
Therefore, we may write our generic representation $\pi= \tau_{\psi_{np}} \rtimes \pi_{gp}$.

We know if $\tau \rtimes \sigma$ is irreducible and generic, then $\tau$ and $\sigma$ are necessarily generic. As a consequence, any generalized Speh representation in the product of $\tau_{\psi_{np}}$ and $\pi_{gp}$ is generic. Our argument in the good parity case then implies $\psi_{gp}$ is tempered.

On the other hand, by the classification of the generic dual of $\GL_d(F)$, each generalized Speh representation in the product of $\tau_{\psi_{np}}$ is generic if and only if all the segments in its Langlands data are not linked, i.e., it contains only one column. In other words, for any $i$, we have 
\[ \frac{a_i-b_i}{2}= \frac{a_i+b_i}{2}-1, \]
which implies $b_i=1$, and hence $\psi_{np}$ is also tempered. This completes the proof of Part (i).

Part (ii) follows directly from Theorem \ref{thm tempered one A-packet}. This completes the proof of the theorem and the proof of the enhanced Shahidi conjecture. 
\end{proof}

Finally, we remark that the proof of Part (2) Conjecture \ref{Enhanced Shahidi conjecture intro}(2) is similar to that of Part (1) as given in Theorem \ref{proof of enh sha conj}, which we omit and just record the statements here.

\begin{thm}\label{thm non-unitary Enhanced Shahidi}
Let $G_n$ be a symplectic or split odd special orthogonal group.
\begin{enumerate}
    \item [(i)] Suppose $\sigma \in \Pi_{\psi}$ for some $\psi \in \Psi^+(G_n)$. Then $\sigma$ is generic only if $\psi$ is generic. Hence, Conjecture \ref{Enhanced Shahidi conjecture intro}(2) holds for $G_n$.
    \item [(ii)] Suppose that $\pi$ is generic and $\pi\in\Pi_\psi$ for some $\psi\in\Psi^{+}(G_n).$ Then $\Psi(\pi)=\{\psi\}.$
\end{enumerate}
\end{thm}

We end this section with the following remark on Steinberg representations of $G_n$ and the proof of the enhanced Shahidi conjecture for Steinberg representations given in \cite[Proposition 8.2]{MS20}.

\begin{remark}\label{rmk Steinberg}\ 
\begin{enumerate}
    \item Let $G$ be a quasi-split group over $F$ with a Borel subgroup $B$. Denote $\mathbf{1}_P$ the trivial representation of $P$. The Steinberg representation $\St_{G}$ of $G$ is the unique irreducible subquotient of $\mathrm{ind}_{B}^G \mathbf{1}_B$ (non-normalized induction) that is not a subquotient of $\mathrm{ind}_{P}^G \mathbf{1}_P$ for any other parabolic subgroup $P$ of $G$. It is known that $\St_{G}$ is square-integrable, generic, and its Aubert-Zelevinsky dual is the trivial representation.  

For $G_n$, we set $N=2n$ if $G_n= \SO_{2n+1}(F)$ and $N=2n+1$ if $G_n= \Sp_{2n}(F)$. Let $\rho$ be the trivial representation of $W_F$ and 
\[ \phi_{\St}:= \rho \otimes S_{N}.\]
Then the Steinberg representation $\St_{G_n}$ is isomorphic to $\pi(\phi_{\St}, 1)$, the unique representation in the tempered $L$-packet $\Pi_{\phi_{\St}}$. Indeed, there is an injection
\[ \pi(\phi_{\St}, 1) \hookrightarrow \rho|\cdot|^{\half{N-1}} \times \rho|\cdot|^{\half{N-3}} \times \cdots \times \rho|\cdot|^{\half{N+1}-n}\rtimes \mathbf{1}_{G_0}, \]
which shows that $D(\pi(\phi_{\St},1))\neq 0$ (see \cite[Lemma 2.5]{HLL22}), where
\[ D:=D_{\rho|\cdot|^{\half{N-1},\dots,\half{N+1}-n}},\]
and $\pi(\phi_{\St},1)$ is a subquotient of $\mathrm{ind}_{B}^{G_n} \mathbf{1}_{B}$. On the other hand, $D(\mathrm{ind}_{P}^{G_n} \mathbf{1}_P)=0$ for any other parabolic subgroup $P$ of $G_n$ by Lemma \ref{lem Leibniz rule} and the fact that 
\[ D_{\rho|\cdot|^{x}}(\mathbf{1}_{G_m})=0\]
for any $x>0$ and $m \in \Z_{>0}$. Therefore, $\pi(\phi_{\St},1)= \St_{G_n}$. 

\item The enhanced Shahidi conjecture implies $\Psi(\St_{G_n})$ is a singleton. This is already proven by Magaard and Savin $($\cite[Proposition 8.2]{MS20}$)$ using a different approach based on the following property of $\phi_{\St}$. Let $\psi_{\St}= \phi_{St} \otimes S_1$. We have
\[ \{ \psi \ | \ \psi^{\Delta}= (\psi_{\St})^{\Delta} \}= \{\psi_{\St}, \widehat{\psi_{\St}} \},\]
where $\psi^{\Delta}$ is an $L$-parameter defined by $\psi^{\Delta}(w,x):=\psi(w,x,x)$. Thus \cite[Corollary 4.2]{Moe09a} implies $\Psi(\St_{G_n}) \subseteq  \{\psi_{\St}, \widehat{\psi_{\St}} \}$. However, since the trivial representation $\mathbf{1}_{G_n}$ is not tempered, it is not in $\Pi_{\psi_{\St}}$, and hence $\St_{G_n}= \widehat{\mathbf{1}_{G_n}}$ is not in $\Pi_{\widehat{\psi_{\St}}}$. Therefore, $\Psi(\St_{G_n})=\{\psi_{\St}\}$.
\end{enumerate}

\end{remark}

\section{Unramified representations of Arthur type}\label{sec unram classification}

In this section, we classify representations of Arthur type of $G_n$ which are unramified with respect to $\mathrm{G}_n(\mathcal{O}_F)$ (Theorem \ref{prop unram L-data}) and consider its applications. These unramified representations are important for global considerations, see \S\ref{sec global}. Therefore, it is desirable to determine all the local Arthur packets to which an unramified representation of Arthur type belongs. We recall an algorithm in \cite{HLL22} to determine whether a representation is of Arthur type or not from its $L$-data in \S \ref{sec algorithm}. Then we proceed the classification in \S \ref{sec classify unram}.

\subsection{An algorithm to determine Arthur type}\label{sec algorithm}
In this subsection, we review an algorithm (\cite[Algorithm 7.9]{HLL22}, see Algorithm \ref{alg Arthur type} below) which determines if a given representation is of Arthur type. We begin by recalling some notation needed for the algorithm.

\begin{defn}[{\cite[Definition 4.1]{HLL22}}]\label{def omega E} \
\begin{enumerate}
    \item Suppose 
\[ \FF=\{([A_i,B_i]_{\rho},l_i,\eta_i)\}_{i \in (I_{\rho, >})}. \]
Define ordered multi-sets  
\begin{align*}
    \Omega(\FF):=& \sum_{i\in I_{\rho}} [A_i,B_i]_{\rho}= \{ \rho|\cdot|^{\alpha_1} ,\dots,\rho|\cdot|^{\alpha_t} \},\\
    \overline{\Omega(\FF)}:=& \sum_{i\in I_{\rho}} [B_i,-A_i]_{\rho}= \{ \rho|\cdot|^{\beta_1} ,\dots,\rho|\cdot|^{\beta_r} \},
\end{align*}
where $\alpha_1 \leq \cdots\leq \alpha_t$ and $\beta_1 \geq \cdots\geq \beta_r$. 
Suppose $\EE=\cup_{\rho} \EE_{\rho}$ is an extended multi-segment. Fix an arbitrary order on the set 
\[ \{ \rho \ | \ \EE_{\rho} \neq \emptyset \}= \{ \rho_1,\dots,\rho_r\}.\]
We define multi-sets $\Omega(\EE)$ and $\overline{\Omega(\EE)}$ to be the sum of multi-sets 
\begin{align*}
    \Omega(\EE)&:= \Omega(\EE_{\rho_1}) +\cdots+ \Omega(\EE_{\rho_r}),\\
    \overline{\Omega(\EE)}&:= \overline{\Omega(\EE_{\rho_1})} +\cdots+ \overline{\Omega(\EE_{\rho_r})}.
\end{align*}

\item For each ordered multi-set $\Omega= \{ \rho_1|\cdot|^{\gamma_1} ,\dots,\rho_t|\cdot|^{\gamma_t} \}$, we define
\begin{align*}
     D_{\Omega}&:= D_{\rho_t|\cdot|^{\gamma_t}} \circ\cdots\circ D_{\rho_1|\cdot|^{\gamma_1}},\\
      S_{\Omega}&:= S_{\rho_1|\cdot|^{\gamma_1}} \circ\cdots\circ S_{\rho_t|\cdot|^{\gamma_t}}.
\end{align*}
For an extended multi-segment $\EE$, we define $D_{\Omega(\EE)}= \circ_{\rho} D_{\Omega(\EE_{\rho})}$. Note that the derivative is independent of the composition order of $D_{\Omega(\EE_\rho)}$.
\end{enumerate}
\end{defn}

Similarly to Definition \ref{def omega E}, we define a multi-set attached to a representation.

\begin{defn}[{\cite[Definition 4.6]{HLL22}}]
For an irreducible representation
\[ \pi= L\left( \Delta_{\rho_1}[x_1,-y_1],\dots,\Delta_{\rho_t}[x_t,-y_t]; \pi(\sum_{j=t+1}^m \rho_j\otimes S_{2z_j+1},\varepsilon) \right), \]
We define
\[ \Omega(\pi):= \{ \rho_1|\cdot|^{x_1},\dots,\rho_t|\cdot|^{x_t} \} + \{ \rho_1|\cdot|^{y_1},\dots, \rho_t|\cdot|^{y_t} \} + \{ \rho_{t+1}|\cdot|^{z_{t+1}},\dots,\rho_m|\cdot|^{z_m} \}. \]
We denote $\Omega(\pi)_\rho$ to be the maximal sub-multi-set of $\Omega(\pi)$ whose elements are of the form $\rho|\cdot|^x$ for some $x \in \R$.
\end{defn}

Now we state an algorithm to determine when a representation is of Arthur type.

\begin{algo}[{\cite[Algorithm 7.9]{HLL22}}]\label{alg Arthur type}
Given a representation $\pi$ of good parity, proceed as follows:
\begin{enumerate}
    \item [\textbf{Step 0}:] 
    Set $\psi=0$. Repeat steps 1 to 4 for each $\rho \in \{ \rho \ | \ \Omega(\pi)_{\rho} \neq \emptyset\}$.
    \item [\textbf{Step 1:}] Let $A= \max\{ x \ | \ \rho|\cdot|^{x} \in \Omega(\pi)_{\rho}\}$ and $\epsilon\in\{0,1/2\}$ such that $A+\epsilon \in \Z$. Set $\Omega^{+}=\emptyset=\Omega^{-}$.
    \item [\textbf{Step 2:}] Compute the following set
    \begin{align*}
        \mathcal{B}=& \{ B >1/2 \ | \ D_{\rho|\cdot|^{B}}(\pi) \neq 0 \}= \{ B_1,\dots,B_r\},
    \end{align*}
    where $B_i$ is decreasing. For each $1 \leq i \leq r$, compute $($by \cite[Theorem 7.1]{AM20}$)$ recursively the integer $k_{i,t}$ and representation $\pi_{i,t}$ for $t$ in the segment $[A+1,B_i-1]$ via 
    \[\begin{cases}
     \pi_{i,B_{i}-1}= \pi,\\
     \pi_{i,t}=D_{\rho|\cdot|^{t}}^{(k_{i,t})}(\pi_{i,t-1}) \text{ is the highest derivative.}
    \end{cases}\]
If $k_{i,A+1} \neq 0$, then $\pi$ is not of Arthur type and the procedure ends. Set $k_{i,t}:=0$ if $t \not\in [A+1,B_i-1]$.
    
    Denote $K_{i,t}:=k_{i,t}-  k_{i-1,t}$. For $t \in [A+1,B_i+1]$, if $ K_{i,t}> K_{i,t-1}$, then $\pi$ is not of Arthur type and the procedure ends. If $K_{i,t}<K_{i,t-1}$, then add $K_{i,t-1}-K_{i,t}$ copies of $\rho \otimes S_{(t-1)+B_i+1}\otimes S_{(t-1)-B_i+1}$ to $\psi$ and add the same copies of elements in the segment $[t-1,B_i]_{\rho}$ in $\Omega^{+}$.
    \item [\textbf{Step 3:}]Reorder
    \[\Omega^{+}=\{\rho|\cdot|^{x_1},\dots,\rho|\cdot|^{x_r}\}\]
    such that $x_1 \leq \cdots\leq x_r$. For $t \in [A+ \epsilon, 1]$, denote
    \[ sh^t(\Omega^{+})=\{\rho|\cdot|^{x_1+t},\dots,\rho|\cdot|^{x_r+t}\}. \]
    Compute the representation $($by \cite[Theorem 7.1]{AM20}$)$
    \[ \pi_{A}= S_{ sh^{A+\epsilon}(\Omega^{+})} \circ\cdots \circ S_{sh^{1}(\Omega^{+})}(\pi),\]
    and the set 
    \[\overline{\mathcal{B}}= \{ B <0 \ | \ D_{\rho|\cdot|^{B}}(\pi_A) \neq 0 \}= \{ \overline{B_1},\dots,\overline{B_{\overline{r}}}\},\]
     where $\overline{B_j}$ is increasing. For each $1 \leq i \leq \overline{r}$, compute $($by \cite[Proposition 6.1]{AM20}$)$ the integer $\overline{k_{i,t}}$ and representation $\overline{\pi_{i,t}}$ for $t$ in the segment $[\overline{B_i}+1,-A-1]$ recursively by 
    \[\begin{cases}
     \overline{\pi_{i,\overline{B_{i}}+1}}= \pi_A, \\
    \overline{\pi_{i,t}}= D_{\rho|\cdot|^{t}}^{(\overline{k_{i,t}})}(\overline{\pi_{i,t+1}}) \text{ is the highest derivative.}
    \end{cases}\]
    If $\overline{k_{i,-A-1}} \neq 0$, then $\pi$ is not of Arthur type and the procedure ends. Set $\overline{k_{i,t}}:=0$ if $t \not\in [\overline{B_{i}}+1,-A-1]$.
    
    Denote $\overline{K_{i,t}}:=\overline{k_{i,t}}-  \overline{k_{i-1,t}}$. For $t \in [\overline{B_i}-1,-A-1]$, if $ \overline{K_{i,t}}> \overline{K_{i,t+1}}$, then $\pi$ is not of Arthur type and the procedure ends. If $\overline{K_{i,t}}<\overline{K_{i,t+1}}$, then add $\overline{K_{i,t+1}}-\overline{K_{i,t}}$ copies of $\rho \otimes S_{-(t+1)+\overline{B_i}+1}\otimes S_{-(t+1)-\overline{B_i}+1}$ to $\psi$, and add the same copies of elements in the segment $[ -t -1,\overline{B_i}]_{\rho}$ in $\Omega^{-}$.
    \item [\textbf{Step 4:}]
    Let $\Omega=\Omega^{+} + \Omega^{-}$. Denote the multiplicity of $\rho|\cdot|^{t}$ in $\Omega\setminus \Omega(\pi)$ (resp. $\Omega(\pi)\setminus \Omega$) by $ m_{1,t}$ (resp. $m_{2,t}$), and let $ M_t =(m_{1,-t-1}-m_{1,t})+m_{2,t}$. 
    
    For any $t \in [A+1, \epsilon+1]$, if $M_{t}>M_{t-1}$, then $\pi$ is not of Arthur type and the procedure ends. If $M_{t}<M_{t-1}$, add $M_{t}-M_{t-1}$ copies of $\rho \otimes S_{(t-1)+\epsilon+1}\otimes S_{(t-1)-\epsilon+1}$ to $\psi$. 
    \item [\textbf{Step 5:}] Construct the local Arthur packet $\Pi_{\psi}$. If there exists $\EE$ in this packet such that $\pi(\EE)=\pi$, then $\pi$ is of Arthur type.  Otherwise, $\pi$ is not of Arthur type. 
    \end{enumerate}
\end{algo}

\subsection{Classification of unramified representations of Arthur type}\label{sec classify unram}
In this subsection, we classify all unramified representations of Arthur type.

Let $\pi$ be a representation of Arthur type and write $\pi= \tau_{nu,>0}\times \tau_{np} \rtimes \pi_{gp} $ by Theorem \ref{thm red from nu to gp}. Then, $\pi$ is unramified if and only if each of $\tau_{nu,>0}, \tau_{np}$ and $\pi_{gp}$ is unramified. Thus, the problem is reduced to the classification of unramified representations of Arthur type and of good parity. We answer this question in the following theorem, which is proved by applying Algorithm \ref{alg Arthur type}. Note that if $\pi \in \Pi_{\psi}$ for some $\psi \in \Psi^+(G_n)$ and $\pi$ is of good parity, then $\psi \in \Psi(G_n)$ automatically by Theorem \ref{thm red from nu to gp} and the definition of good parity (Definition \ref{def good parity reps}). 

\begin{thm}\label{prop unram L-data}
Suppose $\pi$ is an irreducible representation of $G_n$ of good parity, written as 
\[\pi=L(\Delta_{\rho_1}[x_1,y_1],\dots, \Delta_{\rho_f}[x_f,y_f];\pi(\phi,\varepsilon)),\]
where
\[ \phi= \sum_{i=f+1}^m \rho_i \otimes S_{2z_i+1}.  \]
Then $\pi$ is unramified and of Arthur type if and only if the following conditions hold.
\begin{enumerate}
    \item [(i)] $x_i=y_i$ for $1 \leq i \leq f$, $z_i=0$ for $f+1 \leq i \leq m$, and $\rho_i$ are unramified characters for $1 \leq i \leq m$. 
    \item [(ii)] For any $\rho$ and $x \geq 0$, define 
    \[ m_{\rho,x}:= \begin{cases}
    \#\{1 \leq i\leq f\ | \ \rho_i \cong \rho,\  x_i=-x\} &\text{if }x>0,\\
    \#\{f+1 \leq i\leq m\ | \ \rho_i \cong \rho\} &\text{if }x=0.
    \end{cases} \]
    Then for any $\rho$ and $x \geq 0$, $m_{\rho,x+1}\leq m_{\rho,x}$.
    \item [(iii)] The character $\varepsilon$ is trivial.
\end{enumerate}
Moreover, in this case, we have $\pi=\pi(\EE)$ where (set $m_{\rho,x}=0$ for $x <0$)
\begin{align}\label{eq EE unramified}
    \EE= \cup_{\rho} \{([x,-x]_{\rho}, \lfloor x+1/2 \rfloor,1)^{m_{\rho,x-1}-m_{\rho,x}}\}_{x \in \half{1}\Z}.
\end{align}
Here the exponent $m_{\rho,x-1}-m_{\rho,x}$ denotes the multiplicity of the extended segment.
\end{thm}

\begin{proof}
Suppose $\pi$ satisfies all the conditions. Then it follows from Theorem \ref{thm (L)}  that $\pi(\EE)=\pi$, and hence $\pi$ is of Arthur type.
 Condition (i) implies that the $L$-parameter $\phi_{\pi}=\phi_{\psi_{\EE}}$ is unramified. Thus the $L$-packet $\Pi_{\phi_{\psi_{\EE}}}$ contains an unramified representation by the Satake isomorphism (\cite{Sat63}). According to \cite[Theorem 1.5.1(a)]{Art13}, this unramified representation in the local Arthur packet $\Pi_{\psi_{\EE}}$ must correspond to the trivial character in $\widehat{\mathcal{S}}_{\psi}$. On the other hand,  \cite[Theorem 3.6]{Ato20b} confirms that $\pi(\EE)$ is the only representation in $\Pi_{\psi}$ corresponding to the trivial character in $\widehat{\mathcal{S}}_{\psi}$. Therefore, we conclude that this unramified representation is exactly $\pi= \pi(\EE)$.

Next, we show that if $\pi$ is unramified and of Arthur type, then it satisfies Conditions (i), (ii) and (iii). Condition (i) follows from $\phi_{\pi}$ being unramified. Assuming Condition (ii) is already verified, then by \cite[Theorem 9.5]{HLL22}, we have $\pi=\pi(\EE')$, where
\[\EE'= \cup_{\rho} \{([x,-x]_{\rho}, \lfloor x+1/2 \rfloor,\varepsilon_{\rho})^{m_{\rho,x-1}-m_{\rho,x}}\}_{x \in \half{1}\Z},\]
with $\varepsilon_{\rho}:=\varepsilon(\rho\otimes S_1)$ if $m_{\rho,0}>0$ and $\varepsilon_{\rho}:=1$ otherwise. Again \cite[Theorem 3.6]{Ato20b} and \cite[Theorem 1.5.1(a)]{Art13} show that $\pi$ is unramified only if $\varepsilon_{\rho}=1$ for all $\rho$, which verifies Condition (iii). Therefore, it remains to show that $\pi$ is of Arthur type only if Condition (ii) holds. We confirm this by applying Algorithm \ref{alg Arthur type} on $\pi$.

In Step 1 of Algorithm \ref{alg Arthur type}, we have $A= \max\{ -x_i \ | \ 1 \leq i \leq f,\  \rho_i \cong \rho  \}$. Next, it is a direct consequence of Condition (i) and \cite[Theorem 7.1]{AM20} that $D_{\rho|\cdot|^{x}}(\pi) = 0$ for any $x > 0$. This implies that Step 2 in Algorithm \ref{alg Arthur type} is trivial. 

For Step 3, it follows from  \cite[Proposition 6.1]{AM20} that if $\pi$ satisfies Condition (i), then for $0 < x \leq A$,
\[ D_{\rho|\cdot|^{-x}}^{(m_{\rho,x}-m_{\rho,x+1})}(\pi)\]
is the highest derivative. Moreover, the resulting representation also satisfies Condition (i). From this observation, we compute the ingredients in Step 3 of the algorithm as follows.
\begin{enumerate}
    \item [$\oldbullet$] $\overline{\mathcal{B}}= \{ -x\ | \ 0< x \leq A,\ m_{\rho,x}-m_{\rho,x+1}>0\}$.
    \item [$\oldbullet$] $\overline{k_{i,t}}= \max(m_{\rho,t}-m_{\rho,t+1},0)$.
    \item [$\oldbullet$] $\overline{K_{i,t}}= m_{\rho,-\overline{B_i}}-m_{\rho,-(\overline{B_i}-1)}$ if $t= \overline{B_i}$, and $\overline{K_{i,t}}=0$ $t<\overline{B_i}$. 
\end{enumerate}
Thus, at the end of Step 3, we obtain 
\[\psi=\bigoplus_{0< x \leq A} (\rho \otimes S_1 \otimes S_{2x+1})^{\max(m_{\rho,x}-m_{\rho,x+1},0)},\]
and 
\[ \Omega= \sum_{0< x \leq A} ([x,-x]_{\rho})^{\max(m_{\rho,x}-m_{\rho,x+1},0)}.\]
Here in the direct sum and summation, $x$ is assumed to be in $\Z +A$.

Note that for $x \in A+\Z \setminus\{0\}$, the multiplicity of $\rho|\cdot|^x$ in $\Omega$ is given by
\[ \sum_{ |x| \leq y \leq X} \max(m_{\rho,y}-m_{\rho,y+1},0) \geq m_{\rho,|x|},\]
and hence $\Omega(\pi) \setminus \Omega= \{(\rho)^{\max(m_{\rho,0}-m_{\rho,1},0)}\}. $

In Step 4, suppose the contrary that there exists an $x$ such that $m_{\rho,x}< m_{\rho,x+1}$. Note that this implies $x+2\leq A+1$. Choose a maximal $x$ with this property. Then the multiplicities of $\rho|\cdot|^{ \pm (x+2)}$ (resp. $\rho|\cdot|^{ \pm (x+1)}$, $\rho|\cdot|^{\pm x}$) in $\Omega$ and $\Omega(\pi)$ are $m_{\rho,x+2}$ (resp. $m_{\rho,x+1}$, $m_{\rho,x+1}$) and $m_{\rho,x+2}$ (resp. $m_{\rho,x+1}$, $m_{\rho,x}$) respectively. In any case, we have
\[ M_{x+1}=0> -(m_{\rho,x+1}-m_{\rho,x}) =M_{x}, \]
and hence $\pi$ is not of Arthur type. This completes the verification of Condition (ii) and the proof of the theorem.
\end{proof}

The above theorem has several applications. First, it implies the Aubert-Zelevinsky dual of these unramified representations are tempered. We remark that this gives another proof of \cite[Theorem 0-5]{MT11}. Second, we have the following proposition. 

\begin{prop}\label{cor unram dual} 
Conjecture \ref{conj unram intro}(i) holds for $G_n.$ That is, if $\pi$ is unramified of Arthur type, then $\Psi(\pi)=\{\psi\}$ is a singleton, $\pi\in \Pi_{\phi_{\psi}}$, and $\psi$ is anti-generic. 
\end{prop}
\begin{proof}
Suppose $\pi$ is unramified of Arthur type. Then write $\pi= \tau_{nu,>0}\times \tau_{np} \rtimes \pi_{gp}$ as in Theorem \ref{thm red from nu to gp}. Since $\phi_{\pi}$ decomposes as a direct sum of one-dimensional representations, so does $\phi_{\tau_{nu,>0}}= \phi_{\psi_{nu,>0}}$ and $\phi_{\tau_{np}}= \phi_{\psi_{np}}$. This shows that the restrictions of $\psi_{nu,>0}$ and $\psi_{np}$ to Deligne-$\SL_2(\BC)$ are trivial. Thus by Corollary \ref{cor reduction from nu to gp}, it suffices to show $\Psi(\pi_{gp})$ is a singleton consisting of an anti-tempered local Arthur parameter.

Now we apply Theorem \ref{prop unram L-data} to $\pi_{gp}$. Take $\EE$ as in \eqref{eq EE unramified} with $\pi(\EE)=\pi_{gp}$. We have
\[dual(\EE)= \cup_{\rho} \{([x,x]_{\rho},0,\varepsilon_{\rho})^{m_{\rho,x-1}-m_{\rho,x}}\}_{x \in \half{1}\Z},\]
where if $\rho \otimes S_2$ is of the same parity as $\widehat{G}_n$ and $\varepsilon_{\rho}=-1$ otherwise. Note that 
\[ \psi_{dual(\EE)}= \bigoplus_{\rho} \bigoplus_{x \in \half{1}\Z} (\rho \otimes S_{2x+1}\otimes S_1)^{\oplus m_{\rho,x-1}-m_{\rho,x}}, \]
which is tempered, and hence $\widehat{\pi}_{gp}=\widehat{\pi(\EE)}=\pi(dual(\EE))$ is tempered. Moreover, we have $\widehat{\pi}_{gp}= \pi(\phi_{\psi_{dual(\EE)}}, \varepsilon) $, where $\varepsilon(\rho \otimes S_{2x+1})= \varepsilon_{\rho}$. In particular, $\widehat{\pi}_{gp}$ satisfies all of the conditions of Theorem \ref{thm tempered one A-packet}. Hence $\Psi(\widehat{\pi(\EE)})= \{\psi_{dual(\EE)}\}$. This implies that $\Psi(\pi_{gp})=\{\psi_{\EE}\}$, which completes the proof of the proposition.
\end{proof}

\begin{remark}\ 
\begin{enumerate}
    \item [(1)] The most crucial condition we used for an unramified representation $\pi$ is that $\phi_{\pi}$ decomposes as a direct sum of one-dimensional representations. In fact, if $\pi$ is of Arthur type and of good parity and has this property, then the arguments in Theorem \ref{prop unram L-data}(ii) also work. We also obtain $\pi= \pi(\EE)$ with $\EE$ of the form \eqref{eq EE unramified} with possibly some different $\eta_i$ since Part (iii) of the proposition may not hold. However, $\pi(\EE)$ still lies in $\Pi_{\phi_{\psi}}$ for some anti-tempered $\psi$ by \cite[Theorem 9.5]{HLL22}. Then the same arguments in Proposition \ref{cor unram dual} work, and we conclude that $\Psi(\pi)=\{\psi\}$. This can be generalized to an arbitrary representation of Arthur type whose $L$-parameter decomposes as a direct sum of one-dimensional representations.  
    \item [(2)]The Aubert-Zelevinsky dual of an unramified representation $\pi$ of Arthur type may not be always generic. Indeed, suppose $\pi \in \Pi_{\psi}$. Then $\widehat{\pi}$ is generic (with respect to the fixed Whittaker datum) if and only if the character $\varepsilon_{\psi}^{MW/W}$ defined in \cite[Definition 8.1]{Xu17b} is trivial. See \cite[Theorem 1.4]{LLS24}. 
\end{enumerate}
\end{remark}

As another application, we prove the following corollary, which directly implies Part (ii) of Conjecture \ref{conj unram intro} for groups $\mathrm{G}_n$. 

\begin{cor}\label{cor unram at most one}
For any $\psi \in \Psi^+(G_n)$, $\Pi_{\psi}$ contains at most one unramified representation.
\end{cor}

\begin{proof}
By Theorem \ref{thm red from nu to gp}, we may assume $\psi$ is of good parity. Let $\pi_1, \pi_2$ be unramified representations in $\Pi_{\psi}$. Theorem \ref{prop unram L-data} gives $\pi_i = \pi(\EE_i)$ where $\EE_i$ is given by \eqref{eq EE unramified}. Then Proposition \ref{cor unram dual} implies that $\psi= \psi_{\EE_1}= \psi_{\EE_2}$, and hence $\EE_1=\EE_2$. We conclude that $\pi_1=\pi_2$, which proves the corollary.
\end{proof}

\section{Global application}\label{sec global}
In this section, we discuss the global applications of the results in \S \ref{section enhanced Shahidi conjecture}, \ref{sec unram classification}. In particular, we prove Propositions \ref{thm partial Clozel 4 intro}, \ref{thm stronger condition intro} (Propositions \ref{thm partial Clozel 4}, \ref{thm stronger condition} below), and consider a sequence of Clozel's conjectures (Conjectures \ref{conj Clozel 5 intro}, \ref{conj Clozel 2 intro}). These results may be known to experts, but we record them here for the convenience of future references.

\subsection{Global Arthur packets}\label{sec global Arthur packets}

In this subsection, we recall the setting of global Arthur packets for $\mathrm{G}_n$ (\cite[\S 1.5]{Art13}). Then, as an application of the enhanced Shahidi conjecture (Theorems \ref{proof of enh sha conj}, \ref{thm non-unitary Enhanced Shahidi}), we prove Proposition \ref{thm partial Clozel 4 intro}, in particular, for any automorphic representation $\pi \in \Pi_{\psi}$, we show that if $\pi_v$ is generic at one finite place, then $\pi_v$ is generic for almost all finite places 
(see Proposition \ref{thm partial Clozel 4} below). 

Let $k$ be a number field with ring of adeles $\mathbb{A}_k$. Following \cite[\S1.4]{Art13}, a global Arthur parameter $\psi$ is a finite sum of irreducible representations $\mu\otimes S_b$ where $\mu$ is an irreducible cuspidal representation of a general linear group over $\mathbb{A}_k.$ That is, $\psi=\boxplus_{i=1}^m\mu_i\otimes S_{b_i}$ where $\mu_i$ are irreducible cuspidal representations of $\GL_{n_i}(\mathbb{A}_k).$ Let $v$ be a finite place of $k.$ Then 
\[(\mu_i)_v=\times_{j\in I_{\mu_i}}\St(\tau_j,a_j).\]
 Here, $I_{\mu_i}$ is a finite indexing set, $\tau_j$ is an irreducible supercuspidal representation (not necessarily unitary) of a general linear group over $k_v,$ $a_i$ is a positive integer, and $\St(\tau_i,a_i)$ is the unique irreducible subrepresentation of 
 \[\tau_i|\cdot|^{\frac{a_i-1}{2}}\times\cdots\times\tau_i|\cdot|^{\frac{1-a_i}{2}}.\]
 We identify $(\mu_i)_v$ with the representation $\oplus_{j\in I_{\mu_i}}\tau_j\otimes S_{a_j}$ of $W_{k_v} \times \SL_2(\BC)$. By an approximation of the Ramanujan conjecture, we may assume that $\tau_j=\tau_j'|\cdot|^x$ where $\tau_j'$ is an irreducible unitary supercuspidal representation of a general linear group over $k_v$ and $\frac{-1}{2}<x<\frac{1}{2}.$ Thus, for any finite place $v$, we associate to $\psi$ a local Arthur parameter by 
 \begin{align}\label{eq global A-par localize}
     \psi_v=\bigoplus_{i=1}^m\left( \bigoplus_{j\in I_{\mu_i}}\tau_j\otimes S_{a_j}\right)\otimes S_{b_i}.
 \end{align} 
Say $\psi$ is a global Arthur parameter of $\mathrm{G}_n(\A_k)$, each local Arthur parameter $\psi_v$ factors through $\widehat{\mathrm{G}}_n(\BC)$ and hence $\psi_v\in\Psi^+(\mathrm{G}_n(k_v))$.

 Arthur constructed global Arthur packets $\Pi_\psi$ in \cite[\S1.5]{Art13}. According to the construction, if $\pi\in\Pi_\psi$ with $\pi=\otimes_v\pi_v$, then $\pi_v\in \Pi_{\psi_v}$ for all places $v$. Moreover, the character associated to $\pi_v$ in $\widehat{\mathcal{S}}_{\psi_{v}}$ is trivial for almost all $v$.

\begin{cor}\label{cor global app unram}
    Suppose that $\pi=\otimes\pi_v$ lies in a global Arthur packet $\Pi_\psi$ of $\mathrm{G}_n$. There exists a finite set of places  $\mathcal{S}$, containing the Archimedean places, such that for $v\not\in\mathcal{S},$ $\pi_v$ is unramified and $\pi_v\in\Pi_{\phi_{\psi_v}}.$
\end{cor}

\begin{proof}
Let $\mathcal{S}$ be a finite set of places containing the Archimedean places such that for $v\not\in\mathcal{S},$ $(\mu_i)_v$ is unramified for any $1 \leq i \leq m$, and the character of $\pi_v$ in $\widehat{\mathcal{S}}_{\psi_v}$ is trivial. Then for $v \not\in \mathcal{S}$, $\psi_{v}$ is anti-generic and $\phi_{\psi_v}$ is unramified. This implies the $L$-packet $\Pi_{\phi_{\psi_v}}$ contains an unramified representation $\pi^{ur}_v$. By \cite[Theorem 1.5.1(a)]{Art13}, $\pi^{ur}_v\in \Pi_{\psi_v}$ corresponds to the trivial character in $\widehat{\mathcal{S}}_{\psi_v}$, and hence $\pi^{ur}_v= \pi_{v}$ since the map $\Pi_{\psi_v} \to \widehat{\mathcal{S}}_{\psi_v}$ is a bijection (by construction of $\Pi_{\psi_v}$ using extended multi-segments or \cite[Lemma 7.1.1]{Art13}). This completes the proof of the corollary.
\end{proof}

Now we prove the following application of the enhanced Shahidi conjecture (Theorems \ref{proof of enh sha conj}, \ref{thm non-unitary Enhanced Shahidi}).

\begin{prop}\label{thm partial Clozel 4}
    Suppose that $\pi=\otimes\pi_v$ lies in a global Arthur packet $\Pi_\psi$ of $\mathrm{G}_n$ and that there exists a finite place $v_0$ such that $\pi_{v_0}$ is generic. Then the following holds:
    \begin{enumerate}
        \item The global Arthur parameter is of the form $\psi=\boxplus_{i=1}^m \mu_i\otimes S_1,$ where $\mu_i$ are irreducible cuspidal representations of general linear groups.
        \item For any place $v,$  $\psi_v$ is generic. In particular, $\pi_{v}$ is generic for almost all finite places $v$.
        \item Furthermore, for almost all places $v$, the  local Arthur parameter $\psi_v$ has trivial restrictions to both the Deligne-$\SL_2(\mathbb{C})$ and Arthur-$\SL_2(\mathbb{C}).$
    \end{enumerate}
\end{prop}

\begin{proof}
    Since $\pi_{v_0}$ is generic of Arthur type, Theorem \ref{thm non-unitary Enhanced Shahidi}(ii) implies $\Psi(\pi_{v_0})$ is a singleton, and hence $ \Psi(\pi_{v_0})=\{ \psi_{v_0}\}$. Then, Theorem \ref{thm non-unitary Enhanced Shahidi}(i) implies $\psi_{v_0}$ is generic, and hence it is of the form
     $$
    \psi_{v_0}= \bigoplus_{\rho}\bigoplus_{i \in I_{\rho}}  \rho \otimes S_{a_i} \otimes S_{1},
    $$
    where $\rho$ is a supercuspidal representation of $\GL(k_v)$. Comparing with \eqref{eq global A-par localize}, we obtain $\psi=\boxplus_{i=1}^m \mu_i\otimes S_1.$ This proves Part (1).

    By Part (1), for each finite place $v$, we have
     \[\psi_v= \bigoplus_{i=1}^m (\mu_i)_v \otimes S_1.\]
     This proves the first statement in Part (2). Since for almost all finite places, the character of $\pi_v$ in $\widehat{\mathcal{S}}_{\psi_{v}}$ is trivial, $\pi_v$ is generic by Theorems \ref{thm Arthur generic}, \ref{thm generic sequence}. 
     This proves the second statement of Part (2). 
     For almost all places $v$, $(\mu_i)_v$ are unramified for all $1 \leq i \leq m$, and hence the restriction of $(\mu_i)_{v}$ to the Deligne-$\SL_2(\BC)$ is trivial for these places. This shows Part (3) and completes the proof of the proposition. 
\end{proof}

We expect that Part (2) of Proposition \ref{thm partial Clozel 4} is true for general quasi-split reductive groups, see Conjecture \ref{conj generic almost all intro}.

\subsection{Clozel's conjectures}\label{sec Clozel}
In this subsection, we consider the list of Clozel's conjectures in \cite{Clo07} on understanding the local components of automorphic representations in the discrete spectrum of reductive groups. As applications of our results in \S \ref{sec unram classification} on the characterization of unramified representations of Arthur type, for groups $\mathrm{G}_n$, we show that the Ramanujan conjecture of $\GL_n$ implies Conjecture \ref{conj Clozel 2 intro} (see Proposition \ref{thm conj Clozel 2} below). We also prove Proposition \ref{thm stronger condition intro} (see Proposition \ref{thm stronger condition} below). 

Let $k$ be a number field and $\mathrm{G}$ be a reductive group defined over $k$. Let
\[\mathcal{A}_{\mathrm{G}, disc} = L^2_{disc}(\mathrm{G}(k)\backslash \mathrm{G}(\A_{k}), \omega),\]
where $\omega$ is a unitary character of  
$Z(k) \backslash Z(\A_k)$, $Z$ is the center of $\mathrm{G}$.  
First, we recall Arthur's classification for the discrete automorphic spectrum for quasi-split classical groups.

\begin{thm}[{\cite[Theorem 1.5.2]{Art13}}]\label{thm global Arthur}
   If $\mathrm{G}$ is a symplectic group or quasi-split special orthogonal group, then 
   \[L^2_{disc}(\mathrm{G}(k)\backslash \mathrm{G}(\A_{k})) = \bigoplus_{\psi} \bigoplus_{\pi \in \Pi_{\psi}} m_{\pi}\pi,\]
   where $\psi$ runs through all global Arthur parameters of $\mathrm{G}(\A_k)$, $\Pi_{\psi}$ is the global Arthur packets, and $m_{\pi}$ is the multiplicity. 
\end{thm}

Next, we recall the Ramanujan conjecture for $\GL_n$.
\begin{conj}\label{conj Ramanujan}
    If $\pi$ is a cuspidal automorphic representation of $\GL_n(\A_k)$, then $\pi_{v}$ is tempered for any place $v$.
\end{conj}

The following is a direct consequence of the above conjecture. Recall that for each place $v$, $\Psi(\mathrm{G}(k_v))$ is the collection of local Arthur parameter of $\mathrm{G}(k_v)$ whose restriction to Weil group $W_{k_v}$ is bounded. 

\begin{lemma}\label{lem Ramanujan}
    Assume Conjecture \ref{conj Ramanujan}. Then any local Arthur parameter of $\mathrm{G}(k_v)$ lies in $\Psi(\mathrm{G}(k_v))$. 
\end{lemma}
\begin{proof}
    Suppose a local Arthur parameter $\psi_v$ comes from the localization of $\psi=\boxplus_{i=1}^m \mu_i \otimes S_{b_i}$. Assuming Conjecture \ref{conj Ramanujan} holds, then $(\mu_{i})_v$ is tempered, and hence its $L$-parameter has bounded image when restricted to $W_{k_v}$. Thus the restriction of 
    \[ \psi_{v}= \bigoplus_{i=1}^m (\mu_{i})_v \otimes S_{b_i}\]
    to $W_{k_v}$ is also bounded. This completes the proof of the lemma.
\end{proof}

Here is the main result of this subsection, which is an application of the results in \S \ref{sec unram classification}.

\begin{prop}\label{thm conj Clozel 2}
    Let $k$ be a number field and $\pi=\otimes_v \pi_v$ be an automorphic representation of $\mathrm{G}_n(\mathbb{A}_k)$ in the discrete spectrum. Assume Conjecture \ref{conj Ramanujan}. Then the following holds. 
\begin{enumerate}
    \item For any finite place $v_0$ such that $\pi_{v}$ is unramified, the Satake parameter of $\pi_v$ is of the form $\phi_{\psi_v}(\mathrm{Frob}_{v}) $ for some $\psi_v \in \Psi(\mathrm{G}(k_v))$. Moreover, write the multi-set of the absolute value of eigenvalues of $\phi_{\psi_v}(\mathrm{Frob}_v)$ as $\{q_v^{w_{v,1}},\ldots, q_v^{w_{v, N}}\}$, where $q_v$ is the cardinality of the residue field of $k_{v}$. Then the multi-set $\{w_{v,1},\ldots, w_{v, N}\}$ is independent of the unramified place $v$.
    \item If there exists a finite place ${v_0}$ such that $\pi_{v_0}$ is unramified and tempered, then every component of $\pi$ is tempered.
\end{enumerate}
\end{prop}

\begin{proof}
 Theorem \ref{thm global Arthur} implies that $\pi$ lies in a global Arthur packet $\Pi_{\psi}$, and hence $\pi \in \Pi_{\psi_v}$, where $\psi_v$ is the localization of $\psi$. Write $\psi= \boxplus_{i=1}^m \mu_i\otimes S_{b_i}$ and 
    \begin{align}\label{eq psi_v}
        \psi_v=\bigoplus_{i=1}^m (\mu_i)_v \otimes S_{b_i}= \bigoplus_{i=1}^m \left(\bigoplus_{j \in I_{\mu_i,v}}\tau_j \otimes S_{a_j} \right) \otimes S_{b_i} 
    \end{align}

For Part (1), suppose $\pi_v$ is unramified. Then since $\pi_v \in \Pi_{\psi_v}$, Proposition \ref{cor unram dual} implies that $\pi \in \Pi_{\phi_{\psi_v}}$ and $\phi_{\psi_v}$ is unramified and hence $\phi_{\psi_v}(\mathrm{Frob}_{v})$ gives the Satake parameter of $\pi_v$. Note that assuming Conjecture \ref{conj Ramanujan}, $\psi_v \in \Psi(\mathrm{G}_n(k_v))$ by Lemma \ref{lem Ramanujan}. Next, since $\phi_{\psi_v}$ is unramified, in the right hand side of \eqref{eq psi_v}, we have $a_j=1$ for any $j \in I_{\mu_i,v}$. On the other hand, assuming Conjecture \ref{conj Ramanujan}, Lemma \ref{lem Ramanujan} implies that the eigenvalues of $\tau_j(\mathrm{Frob}_v)$ all have absolute value one. We conclude that for any unramified place $v$,
    \begin{align*}
        \{{w_{v,1}},\ldots, {w_{N,1}}\}&= \sum_{i=1}^m \sum_{j \in I_{\mu_i,v}}{\dim(\tau_j)} \left(\left\{ \half{b_i-1},\half{b_i-3},\ldots, \half{1-b_i} \right\}\right)\\
        & =\sum_{i=1}^m \dim(\mu_i) \left(\left\{ \half{b_i-1},\half{b_i-3},\ldots, \half{1-b_i} \right\}\right),
    \end{align*}
which is independent of the place $v$. Here $\dim(\tau_j)$ and $\dim(\mu_i)$ denotes the multiplicities. This completes the proof of Part (1).

For Part (2), suppose $\pi_{v_0} \in \Pi_{\psi_{v_0}}$ is both unramified and tempered. We have $\Psi(\pi_{v_0})=\{\psi_{v_0}\}$ since $\pi_{v_0}$ is unramified. On the other hand, $\psi_{v_0}$ must be tempered since $\Psi(\pi_{v_0})$ has a tempered member. Next, by \cite[Theorem 1.5.1(a)]{Art13}, the unramified representation $\pi_{v_0}$ must correspond to the trivial character in $\widehat{\mathcal{S}}_{\psi_{v_0}}$. Therefore, $\pi_{v_0}$ is generic by Theorem \ref{thm Arthur generic}. Now Proposition \ref{thm partial Clozel 4}(1) implies that $\psi=\boxplus_{i=1}^m \mu_i\otimes S_1$, and hence for any place $v$,
    \[ \psi_{v}= \bigoplus_{i=1}^m (\mu_i)_v \otimes S_1 \]
    is generic. If we further assume Conjecture \ref{conj Ramanujan}, then the restriction of $(\mu_i)_{v}$ to $W_{k_v}$ is bounded for any $i$, and hence the restriction $\psi_{v} $ to $W_{k_v}$ is also bounded. We conclude that for any place $v$, $\psi_{v}$ is tempered and so is $\pi_{v} \in \Pi_{\psi_v}$. This completes the proof of Part (2) and th proposition.
\end{proof}

 We remark that we did not assume Conjecture \ref{conj Ramanujan} throughout the paper except in the above proof. Especially, it is not assumed in the proof of Proposition \ref{thm partial Clozel 4}(2). On the other hand, if we do assume Conjecture \ref{conj Ramanujan}, then we obtain a stronger conclusion (which directly implies Conjecture \ref{conj Clozel 5 intro}), comparing to Proposition \ref{thm partial Clozel 4}(2).

\begin{prop}\label{thm stronger condition}
    Assume Conjecture \ref{conj Ramanujan}. Suppose that $\pi=\otimes\pi_v$ lies in a global Arthur packet $\Pi_\psi$ of $\mathrm{G}_{n}(\A_k)$ and that there exists a finite place $v_0$ such that $\pi_{v_0}$ is generic. Then $\pi_v$ is tempered for all places $v$. 
\end{prop}
\begin{proof}
    By Proposition \ref{thm partial Clozel 4}(1), the global Arthur parameter is of the form $\psi= \boxplus_{i=1}^m \mu_i \otimes S_1$, and hence for any place $v,$ the local Arthur parameter $\psi_v$ is trivial on the Arthur-$\SL_2(\BC)$. On the other hand, assuming Conjecture \ref{conj Ramanujan}, Lemma \ref{lem Ramanujan} implies that $\psi_v$ has bounded image when restricted to $W_F$. We conclude that for any place $v$, $\psi_v$ is tempered, and hence $\pi_{v} \in \Pi_{\psi_v}$ is also tempered. This completes the proof of the proposition.
\end{proof}

\end{document}